\documentclass[12pt]{article}

\usepackage[english]{babel}
\usepackage{amsmath,amsthm}
\usepackage{amsfonts}
\usepackage{graphicx}
\usepackage{epsfig}
\usepackage{epstopdf}
\usepackage{url}
\usepackage{amssymb}
\usepackage{bbm}
\usepackage{mathrsfs}
\usepackage{algorithm}
\usepackage{algorithmic}

\usepackage{color}
\usepackage[makeroom]{cancel}
\usepackage{tikz-cd}
\usepackage[titletoc]{appendix}
\usepackage[%
breaklinks%
,colorlinks%
,linkcolor=red
,anchorcolor=blue%
,pagecolor=blue%
,citecolor=blue%
,bookmarks=false%
]{hyperref}
\definecolor{codegreen}{rgb}{0,0.6,0}
\definecolor{codegray}{rgb}{0.5,0.5,0.5}
\definecolor{codepurple}{rgb}{0.58,0,0.82}
\definecolor{backcolour}{rgb}{0.95,0.95,0.92}
\usepackage{listings}
\lstdefinelanguage{Sage}[]{Python}
{morekeywords={False,sage,True},sensitive=true}
\definecolor{dblackcolor}{rgb}{0.0,0.0,0.0}
\definecolor{dbluecolor}{rgb}{0.01,0.02,0.7}
\definecolor{dgreencolor}{rgb}{0.2,0.4,0.0}
\definecolor{dgraycolor}{rgb}{0.30,0.3,0.30}

\definecolor{delim}{RGB}{20,105,176}
\definecolor{numb}{RGB}{106, 109, 32}
\definecolor{string}{rgb}{0.64,0.08,0.08}
\date{}

\newtheorem{theorem}{Theorem}[section]
\newtheorem{corollary}[theorem]{Corollary}
\newtheorem{lemma}[theorem]{Lemma}
\newtheorem{proposition}[theorem]{Proposition}

\newtheorem{remark}[theorem]{Remark}

\newtheorem{example}[theorem]{Example}
\newtheorem*{theorem*}{Theorem}

\def\P{{\mathbb{P}}}
\def\T{{\mathbb{T}}}

\def\S{{\mathcal{S}}}
\def\N{{\mathbb{N}}}

\def\para{\vspace{1.5mm}}

\def\cZ{\mathcal{Z}}

\def\K{\mathbb{K}}

\def\Gr{\mathbb{G}\mathrm{r}}
\def\bbV{\mathbb{V}}
\def\O{\mathcal{O}}
\def\Sd{S^dV}
\def\S{S^3V}
\def\SS{S^2V}
\def\Ss{S^4V}
\def\sub{\mathrm{Sub}}
\def\grad{\P(\langle\nabla F\rangle)}
\def\gradg{\P(\langle\nabla G\rangle)}
\def\pgl{\mathrm{PGL}(2)}
\def\pgn{\mathrm{PGL}(n+1)}
\def\t{\mathrm{t}}
\def\hil{\mathrm{Hilb}}
\def\hessMat{\mathrm{H}}

\title{The Hessian correspondence of hypersurfaces\\ of degree $3$ and $4$}

\author{Javier Sendra-Arranz}
 
\begin{document}
\maketitle

\para

\begin{abstract}
\noindent 
 Let $X$  be a hypersurface, of degree $d$, in an $n$--dimensional projective space. The Hessian map is a rational map from $X$ to the projective space of symmetric matrices that sends a point $p\in X$ to the Hessian matrix of  the defining polynomial of $X$ evaluated at $p$. The Hessian correspondence is the map that sends a hypersurface to its Hessian variety; i.e. the Zariski closure of its image via the Hessian map. In this paper, we study this correspondence for hypersurfaces with Waring rank at most $n+1$ and for hypersurfaces of degree $3$ and $4$. We prove that, for hypersurfaces with Waring rank $k\leq n+1$, the map is birational onto its image for $d$ even, and it is generically finite of degree $2^{k-1}$ for $d$ odd.
 We prove that, for degree $3$ and $n=1$, the map is two to one, and that, for degree $3$ and $n\geq 2$, and for degree $4$, the Hessian correspondence is birational. In this study, we introduce the $k$–gradients varieties and analyze their main properties. We provide effective algorithms for recovering a hypersurface from its Hessian variety, for degree $3$ and $n\geq 1$, and for degree $4$ and $n$ even.
\end{abstract}

\section{Introduction}

We start motivating the notions of Hessian map and Hessian correspondence.
For this purpose, we first introduce the concept of Gaussian map and polar map.
 In order to be more precise, let $\mathbb{K}$ be an algebraically closed field of characteristic $0$,  let $V:=\K^{n+1}$, for $n\geq 1$, with coordinates $x_0,\ldots,x_n$,  and let $S^dV$  denote the $d$--th symmetric power of $V$. Given a degree $d$ homogeneous polynomial $F\in \P(\Sd)$, the \textsf{Gaussian map} is the rational map 
\begin{equation}\label{eq-Gmap} \begin{array}{cccc}
     g^1: &\bbV(F)&\dashrightarrow&\left(\P^n\right)^{*} \\ \noalign{\vspace*{1mm}}
  &   p &\longmapsto & \nabla F(p).
\end{array}
\end{equation} In other words, the Gaussian map associates to each smooth point of $\bbV(F)$ its tangent hyperplane. 
The \textsf{dual variety} $\bbV(F)^{*}$ of $\bbV(F)$ is the closure of its image via the Gaussian map, namely, the closure of $g^1(\bbV(F))$. For generic $F$, the dual variety  $\bbV(F)^{*}$ is a hypersurface of degree $d(d-1)^{n-1}$ (see \cite{D}). Its defining polynomial $\Delta_F$  is called the \textsf{$F$--discriminant}. In this situation,  the rational map 
\[
\begin{array}{cccc}
I_{d,n}:&\P(S^dV)&\dashrightarrow& \P(S^{d(d-1)^{n-1}}V)\\ \noalign{\vspace*{1mm}} & F&\longmapsto&\Delta_F
\end{array}
\]
is considered. 
By the Biduality Theorem (see \cite{Disc} Theorem 1.1.), the map $I_{d,n}$ is birational onto its image and $F$ can be  recovered from $I_{d,n}(F)$  by taking the dual variety of $I_{d,n}(F)$. 

\para 

The previous construction is based on the Jacobian of $F$. The generalization of the Gaussian map to higher order derivatives is called the degree $k$ polar map. For $k\geq 1$, the \textsf{degree $k$ polar map} is the rational map 
 \begin{equation}\label{eq-kPolar} \begin{array}{cccl}
 g^k: & \bbV(F) &\dashrightarrow & \P(S^k V) \\ \noalign{\vspace*{1mm}}
        & p & \longmapsto & \sum \dfrac{\partial^kF}{\partial x_{i_1}\cdots\partial x_{i_k}}(p)x_{i_1}\cdots x_{i_k}.
\end{array}
\end{equation}
Note that,  $\P(S^1 V)$ can be identified with $\left(\P^n\right)^{*} $, and thus \eqref{eq-kPolar}, with $k=1$, yields \eqref{eq-Gmap}. 
The \textsf{$k$--th polar variety} is defined as the closure of the image of  $g^k$. In \cite{L}, the author proved that if $g^k$ is regular, the $k$--th polar variety has dimension $n-1$ and degree $d(d-k)^{n-1}$.

\para

The Hessian map is introduced using the second order derivatives. Let $\hessMat_F$ denote the Hessian matrix of $F$. Then the \textsf{Hessian map} is defined as the rational map
 \begin{equation}\label{eq-Hess-map} \begin{array}{cccc}
h_F:&\bbV(F)&\dashrightarrow&\P(S^2V) \\ \noalign{\vspace*{1mm}}
 &p & \longmapsto & \hessMat_{F}(p).
 \end{array}
 \end{equation}
Clearly the Hessian map generalizes the Gaussian map and, taking into account that  $\P(S^2 V)$ can be identified with the set of the symmetric matrices, the  {Hessian map} can be seen as the degree $2$ polar map.  In this case, the $2$-nd polar variety is called the \textsf{Hessian variety}. Alternatively, in \cite{H1}, \cite{H2} and \cite{ID}, the {Hessian map} is defined as the map that associates to $F$ the  \textsf{Hessian polynomial},  namely, the determinant of the Hessian matrix. It should be noticed that our notion of Hessian map differs  from the concept introduced in these two papers.

\para

 Let $z_{i,j}$ be the coordinate of $\P(S^2V)$ representing the monomial $x_ix_j$. By abuse of notation, we will often consider $h_F$, in \eqref{eq-Hess-map}, as a rational map from $\P^n$ instead of from $\bbV(F)$.
Let $\Xi_{d,n}$ be the closure of  the set
$
\{(x,F)\in\P(S^2V)\times\P(S^dV):x\in \overline{h_F(\bbV(F))}\}.
$
The projection $\Xi_{d,n}\dashrightarrow\P(S^dV)$ defines a family of algebraic varieties parametrized by an open subset of $\P(S^dV)$. Let $p_{d,n}(t)$ be the Hilbert polynomial of a generic fiber of this projection. Then, we define the \textsf{Hessian correspondence} as the scheme-theoretical morphism 
\begin{equation}\label{eq:Hess corr}
 \begin{array}{cccc}
     H_{d,n}:&\P(S^dV)&\dashrightarrow&\hil^{p_{d,n}(t)}(\P(S^2V))\\
     \noalign{\vspace*{1mm}}
      & F&\longmapsto&\overline{h_F(\bbV(F))}.
     \end{array}
 \end{equation}
 In other words, $H_{d,n}$ is the rational map that sends a  hypersurface to its Hessian variety.
Therefore, $h_F$ and $H_{d,n}$ can be seen as the generalization, to the second order of derivation, of the Gaussian map and $I_{d,n}$, respectively.

\para 

The base locus of the Hessian correspondence consists of the polynomials $F$ whose Hessian variety does not have Hilbert polynomial $p_{d,n}(t)$. Analogously, we could define the Hessian correspondence for varieties of polynomials contained in this base locus. 
The main example we are interested in is the case of the variety of polynomials with border rank $k$. The \textsf{Waring rank} of a polynomial $F\in S^dV$ is the minimum integer $k\geq 0$ such that $F = l_1^d+\cdots+l_k^d$ for some linear forms $l_1,\ldots,l_k$.  We denote the Veronese variety of all polynomials of the form $l^d$, for $l\in\P(S^1V)$, by $V^{d,n}$. The closure of the set of degree $d$ polynomials of Waring rank $k$ is the $k$--th secant variety of $V^{d,n}$, denoted by $\sigma_k(V^{d,n})$.  The \textsf{border rank} of a polynomial $F\in\P(S^dV)$ is the minimum integer $k\geq 0$ such $F\in\sigma_k(V^{d,n})$. 
For further literature on Waring rank and secant varieties we refer to \cite{sym}. We define the restriction of the Hessian correspondence to $\sigma_k(V^{d,n})$ as the rational map 
\begin{equation}\label{eq:Hess corr rank}
 \begin{array}{cccc}
     H_{d,n,k}:&\sigma_k(V^{d,n})&\dashrightarrow&\hil^{p_{d,n,k}(t)}(\P(S^2V))\\
     \noalign{\vspace*{1mm}}
      & F&\longmapsto&\overline{h_F(\bbV(F))}.
     \end{array}
 \end{equation}
where $p_{d,n,k}(t)$ is the Hilbert polynomial of the Hessian variety of a generic hypersurface in $\sigma_k(V^{d,n})$. Note that, in general, it is not true that $p_{d,n}= p_{d,n,k}$. However, for any $d$ and $n$ there exists a $k$ such that $\sigma_k(V^{d,n})=\mathbb{P}(S^dV)$, and $H_{d,n}=H_{d,n,k}$.

\para
 
The natural question, and central aim of this paper, is the extension of the  results concerning the maps $g^1$ and $I_{d,n}$ to $h_F$, $H_{d,n}$ and $H_{d,n,k}$. More concretely, given a generic variety $X$ in the image of $H_{d,n}$ ($H_{d,n,k}$, respectively), our goal is to recover, when possible, the polynomials whose Hessian variety is $X$; that is, to determine $F\in \P(S^dV)$ such that $H_{d,n}(F)=X$. In order words, we study the description of the  fiber $H_{d,n}^{-1}(X)$.
Some of the questions that we aim to answer concerning the fibers of $H_{d,n}$ (and $H_{d,n,k}$ respectively) are: Is $H_{d,n}$ birational onto its image? Is $H_{d,n}$ finite? Is the fiber a unique polynomial up to change of coordinates in $\P^n$? Or up to isomorphisms or birationality of the corresponding hypersurfaces?

\para 

We observe that the cases $d=1$ and $d=2$ have a direct answer. For $d=1$, the problem is trivial since the Hessian is the zero matrix. Let now $d=2$. For $F\in\P(S^2V)$, the Hessian matrix $H_F$ is a symmetric matrix with constant entries. Therefore,  $H_{2,n}$  is the map $H_{2,n}:\P(S^2V)\rightarrow\P(S^2V)$ that sends $F$ to $h_F$. Applying the Euler formula twice, we get that $2F = (x_0,\ldots,x_n)h_F(x_0,\ldots,x_n)^{\mathrm{t}}$.
In particular, we get that $H_{2,n}$ is the identity map.  Therefore, the interesting questions arises for $d\geq3$.

\para 

In this paper we analyze the above questions for polynomials with Waring rank at most $n+1$ for polynomials of degree $3$ and $4$. To study $H_{d,n,k}$ for $k\leq n+1$, we analyze the relation between the $\mathrm{PGL}(n+1)$ action on $\P^n$ and the Hessian map. 
Whereas, the main idea for approaching the cases $d=3$ and $d=4$ is to study $h_F(\P^n)$ instead of $H_{d,n}(F)$. We prove that for $d=3$, $h_F(\P^n)$ is the smallest linear subspace containing $H_{3,n}(F)$. For $d=4$ and $n\geq 2$, in Proposition \ref{prop:uniq vero} we prove that $h_F(\P^n)$ is the unique Veronese variety containing $H_{4,n}(F)$. Let us, at this point, highlight the connection between the variety $h_F(\P^n)$ and the classical  \textsf{Hesse problem}. The \textsf{$(d,n)$--Gordan-Noether locus} is the locus of polynomials in $\P(S^dV)$ whose Hessian polynomial vanishes. The Hesse problem deals with the description of the $(d,n)$--Gordan-Noether locus. For $n\leq 3$, a polynomial lies in this locus if and only if it is a cone (see \cite{GN}). We refer to \cite{CRS} for further information on the Gordan-Noether locus, and to \cite{BFP} for some recent contributions. In our setting, a polynomial $F$ lies in the $(d,n)$--Gordan-Noether if and only if $h_F(\P^n)$ is contained in the hypersurfaces of symmetric matrices with vanishing determinant. We expect that our study of $h_F(\P^n)$ could enlight new ideas for approaching the Hesse problem.
 
\para

This paper fully answers the above questions for polynomials with Waring rank at most $n+1$ and for polynomials of degree $3$ and $4$. The main contributions of this paper are: 

\para 

   \begin{theorem*}[Theorem \ref{theo:sym hessian}]
       For $d\geq 3$ and $k\leq n+1$, $H_{d,n,k}$ is generically finite. For $d$ odd, it has degree $2^{k-1}$. For $d$ even, $H_{d,n,k}$ is birational.
   \end{theorem*}

\para 

\para 
 
\begin{theorem*}[Theorem \ref{theo:inj H3 1}]
    For $n\geq 2$, $H_{3,n}$ is birational onto its image.
\end{theorem*}

\para 

\para

\begin{theorem*}[Theorem \ref{theo:4,n}]
For $n\geq 1$, $H_{4,n}$ is birational onto its image.
\end{theorem*}

\para 

In addition, 
we provide effective algorithms for recovering a hypersurface from its Hessian variety 
 for $d=3$ and $n\geq 1$ (see Remark \ref{rem:fiberH31} and Algorithm \ref{alg:1} in Section \ref{sec:h3n>1}), and for $d=4$ and $n$ even (see Remark  \ref{rem:rec H41} and Algorithm  \ref{alg:2} in Section  \ref{subsec:alg d=4}). Moreover, as a consequence of Algorithm \ref{alg:2}, we provide a method for computing a parametrization of a Veronese variety of even dimension arising from a quadratic Veronese embedding (see Remark \ref{rem:para}).

\para 

The paper is structured as follows. In Section \ref{sec:sym}, we study the Hessian correspondence for polynomials with Waring rank at most $n+1$.
We prove that 
$H_{d,n,k}$ for $k\leq n+1$ is birational onto its image for $d$ even. For $d$ odd we show that a generic fiber of $H_{d,n,k}$ consists of $2^{k-1}$ points. 
In particular, we show that $H_{3,1}$ is a two to one map.
 In Section \ref{sec:d=3} we study $H_{3,n}$ and some related varieties. In Subsection \ref{sec 3,1} we provide a method for computing the fibers of $H_{3,1}$. We construct the rational involution of $\P(\S)$ that preserves the preimages of $H_{3,1}$ and we describe the variety of $3$ plane points that can be obtained from $H_{3,1}$. 
 In Subsection \ref{sec:var k planes} we focus on the rational map
\begin{equation}\label{eq:alpha first}\begin{array}{cccc}
\alpha_n:&\P(\S)&\dashrightarrow&\Gr(n,\P(\SS))\\
 & F&\longmapsto& h_F(\P^n),\end{array}
\end{equation}
and we introduce  the \textsf{variety  of $k$--gradients } $\phi_k\subseteq \mathbb{G}\mathrm{r}(k,\P(S^2V))$ as the variety of $k$--planes $\Gamma$ containing all the first order derivatives of a polynomial $F\in\P(S^3V)$.
In Subsection \ref{sec:alphak} we prove that $\alpha_n$ is birational onto its image.
In Subsection \ref{sec:h3n>1} we use the map $\alpha_n$ to 
  prove that $H_{3,n}$
is birational onto its image for $n\geq 2$ and we provide an algorithm for recovering the hypersurface from its Hessian variety. The irreducible components of $\phi_k$ are described in Subsection \ref{sec:irred grad}.

\para 

Section \ref{sec:d=4} is devoted to the case $d=4$. 
In Subsections \ref{sec:41} we show that $H_{4,1}$ is birational onto its image, we provide a method for computing the fibers and we describe the variety of $4$ plane points that can be obtained from $H_{4,1}$. In Subsection \ref{4n>1} we prove that $H_{4,n}$ is birational onto its image for $n\geq 2$, and in Subsection \ref{subsec:alg d=4}  an effective algorithm for recovering $F$ from $H_{4,n}(F)$ when $n$ even is derived.

\para

At the end of the paper we summarize the conclusions and we point out some   open questions.

\section{Hessian correspondence and Waring rank}\label{sec:sym}

We start this section with a motivating example. Let $n=3$ and $d=3$, and consider the polynomials $F = x_0^3+x_1^3+x_2^3$  and $G = x_0^3+x_1^3-x_2^3$. Both $F$ and $G$ have Waring rank $3$ and their Hessian matrices are the diagonal matrices
\[
\begin{pmatrix}
    6x_0&0&0\\0&6x_1&0\\ 0&0&6x_2
\end{pmatrix}
\text{ and } 
\begin{pmatrix}
    6x_0&0&0\\0&6x_1&0\\ 0&0&-6x_2
\end{pmatrix}
\]
respectively. In particular, the corresponding Hessian maps $h_F$ and $h_G$ factors through the closed embedding of $\bbV(z_{01},z_{02},z_{12})$ in $\P(S^2V)$. Moreover,  one can check that both the Hessian variety of $F$ and $G$ are equal to 
\[
\bbV(z_{01},z_{02},z_{12},z_{00}^3+z_{11}^3+z_{22}^3)\subset\P(S^2V)\simeq\P^5.
\]
Similarly, the Hessian variety of $ x_0^3-x_1^3+x_2^3$ and $-x_0^3+x_1^3+x_2^3$  equals to $H_{3,2}(F)$ too.  In particular, we deduce that there exists a locus in $\P(S^3V)$ where the $H_{3,2}$  is not injective but has degree $4$. 
In this section we study how the Hessian correspondence acts on the locus of these types of polynomials.

\para 

In \eqref{eq:Hess corr}, we defined the Hessian correspondence as a rational map from $\P(S^dV)$ to certain Hilbert scheme. Similarly, in \eqref{eq:Hess corr rank}, we defined the restriction $ H_{d,n,k}$ of the Hessian correspondence to $\sigma_k(V^{d,n})$. Note that the target spaces of $H_{d,n}$ and $H_{d,n,k}$ might be distinct Hilbert schemes. 
For instance, in Section \ref{sec:d=4} we will see that the Hessian variety of a generic polynomial of degree $4$ is the intersection of a Veronese variety in $\P(S^2V)$ with a quadratic hypersurface. On the other hand, we will see that the Hessian variety of a generic element in $\sigma_{n+1}(V^{4,n})$ is the intersection of a quadratic hypersurface with an $n$ dimensional linear subspace. One can check that both Hessian varieties have different Hilbert polynomial. Therefore, the maps $H_{d,n}$ and $H_{d,n,k}$ are distinct since their target spaces differ.

\para 

In this section we study the rational map $H_{d,n,k}$ for $k\leq n+1$. 
To do so, for any $d$ we consider the  group action of $\pgn$ on $\P(\mathrm{S}^{d}V )$ given  by  $g\cdot F=F\circ g^{\mathrm{t}}$ for $g\in\pgn$ and $F\in\P(\mathrm{S}^{d}V )$.
This defines a representation $\rho:\mathrm{PGL}(n+1)\rightarrow\mathrm{PGL}(S^dV)$. Since the target space of the Hessian map is $\P(S^2V)$, the representation $\rho$ for $d=2$ will play a fundamental role. The next lemma analyzes how this action interacts with the Hessian map.

\para 

\begin{lemma}\label{lemma:n=1 2} Let $n\geq 1$ and $d\geq 3$. Then,
 for $F\in \P(\mathrm{S}^{d}V)$ and $g\in\mathrm{PGL}(n+1)$, we have that $\rho(g^{\mathrm{t}})^{\mathrm{t}}\circ h_F = h_{g\cdot F}\circ (g^{\mathrm{t}})^{-1}$. In particular, we have that $H_{d,n}(g\cdot F) = \rho(g^{\mathrm{t}})^{\mathrm{t}}(H_{d,n}(F))$.
\end{lemma}
\begin{proof}
By the Leibniz Rule, for $i,j\in[n]$ we have that 
$
\frac{\partial^2 g\cdot F }{\partial x_i\partial x_j} = \sum_{k,l} 
 g^{\mathrm{t}}_{k,i}
 g^{\mathrm{t}}_{l,j}
\frac{\partial^2  F }{\partial x_k\partial x_l}\circ  g^{\mathrm{t}}.
$
Hence, the coordinate $z_{i,j}$ of the composition $h_{ g^{\mathrm{t}}\cdot F}\circ (g^{\mathrm{t}})^{-1}$ is
\begin{equation}\label{eq:diag1 proof}
\displaystyle
 \sum_{k,l} 
 g_{i,k}
 g_{j,l}
\frac{\partial^2  F }{\partial x_k\partial x_l}=  \sum_{k<l} 
 (g_{i,k} g_{j,l}+g_{i,l} g_{j,k})
\frac{\partial^2  F }{\partial x_k\partial x_l}+
 \sum_{k} 
 g_{i,k}
 g_{j,k}
\frac{\partial^2  F }{\partial x_k^2}
.
\end{equation}
On the other hand, 
  the coordinate $z_{i,j}$ of $\rho(g^{\mathrm{t}})^{\mathrm{t}}$  is 
 \[
 \displaystyle
\sum_{k<l}(g_{i,k}g_{j,l}+g_{i,l}g_{j,k})z_{k,l}+
 \sum_{k}g_{i,k}g_{j,k}z_{k,k}
 .
 \]
By \eqref{eq:diag1 proof}, the coordinate $z_{i,j}$ of $\rho(g^{\mathrm{t}})^{\mathrm{t}}\circ h_F$ is equal to the coordinate $z_{i,j}$ of $h_{ g\cdot F}\circ (g^{\mathrm{t}})^{-1}$. 

The equality $H_{d,n}(g\cdot F) = \rho(g^{\mathrm{t}})^{\mathrm{t}}(H_{d,n}(F))$ follows from the first part of the lemma and the fact that  $\bbV(g\cdot F) = (g^{\mathrm{t}})^{-1}(\bbV(F))$. 
\end{proof}

As a consequence of Lemma \ref{lemma:n=1 2}, we derive the following result. 

\para 

\begin{proposition}\label{prop:isofibers}
    Let $F\in\P(S^dV)$ and $g\in \mathrm{PGL}(n+1)$. Then, $g\cdot H_{d,n}^{-1}(H_{d,n}(F))=H_{d,n}^{-1}(H_{d,n}(g\cdot F))$. In particular the fibers of $H_{d,n}(F)$ and $H_{d,n}(g\cdot F)$ via $H_{d,n}$ are isomorphic.
\end{proposition}
\begin{proof}
    Let $\tilde{F}\in\P(S^dV)$ with same Hessian variety as $F$. By Lemma \ref{lemma:n=1 2}, we get that 
    \[
    H_{d,n}(g\cdot \tilde{F }) = \rho(g^\mathrm{t})^\mathrm{t}( H_{d,n}(\tilde{F}))=\rho(g^\mathrm{t})^\mathrm{t}( H_{d,n}(F)) =  H_{d,n}(g\cdot F).
    \]
    In particular, $g\cdot \tilde{F}$ has same Hessian variety as $g\cdot F$. Hence, $g\cdot H_{d,n}^{-1}(H_{d,n}(F))$ is contained in $H_{d,n}^{-1}(H_{d,n}(g\cdot F))$. Similarly, since $g^{-1}\cdot g\cdot F= F$, we deduce that $g^{-1}\cdot H_{d,n}^{-1}(H_{d,n}(g\cdot F))$ is contained in $ H_{d,n}^{-1}(H_{d,n}(F))$. Now, the proof follows from applying $g$ to the previous inclusion.
\end{proof}

Note that Lemma \ref{lemma:n=1 2} and Proposition \ref{prop:isofibers} also hold if we replace $H_{d,n}$ by $H_{d,n,k}$ for any $k$. Now, assume that $k\leq n+1$ and let 
$\mathcal{U}_k\subseteq \sigma_k(V^{d,n})$ be the open subset of polynomials of the form $l_1^d+\cdots+l_{k}^d$ such that $l_1,\ldots,l_{k}\in S^1V$ are linearly independent. Note that $\mathcal{U}_k$ is a nonempty dense open subset of $\sigma_k(V^{d,n})$.  Moreover, since $k\leq n+1$, it is the $\mathrm{PGL}(n+1)$ orbit of the polynomial $F_k:= x_0^d+\cdots+x_{k-1}^d$. 
Therefore, by Proposition \ref{prop:isofibers}, in order to understand the fibers of $H_{d,n,k}$ restricted to $\mathcal{U}_k$, it is enough to study the fiber of the Hessian variety of the polynomial $F_k:=x_0^d+\cdots+x_{k-1}^d$. 

\para

Note that up to scalar multiplication, the Hessian matrix of $F_k$ is the diagonal matrix 
\[
\begin{pmatrix}
    x_0^{d-2}& & & & &\\
    & \ddots & & & & \\
     & & x_{k-1}^{d-2}& & &\\
     & & &0 & &\\
    & & & &\ddots &\\
    & & & & &0\\
\end{pmatrix}.
\]
In particular, the Hessian variety of $F_k$ is a hypersurface in the $k$--th dimensional linear subspace 

\begin{equation}\label{eq:proj sym}
\P(U):=\bbV(z_{i,j}:(i,j)\not\in \{(0,0),\ldots,(k-1,k-1)\})\subset\P(S^2V).
\end{equation}

Now, let $G\in \P(S^dV)$ be such that $F_k$ and $G$ have the same Hessian variety. This implies that $\frac{\partial^2 G}{\partial x_ix_j}= 0 $
 for $(i,j)\not\in \{(0,0),\ldots,(k-1,k-1)\}$. From the Euler's formula we deduce that the first order derivatives of $G$ with respect to $x_{k},\ldots,x_n$ vanish. Hence, $G$ is a polynomial in the variables $x_0,\ldots,x_{k-1}$. Therefore, in order to study $H_{d,n,k}$, we assume that $k=n+1$.

\para 

\begin{lemma}\label{lemma:sym fibers}
Let $G\in \P(S^dV)$ with the same Hessian variety as $F_{n+1}$. Then, $G = \lambda_0x_0^d+ \cdots+\lambda_nx_n^d$ for $\lambda_0,\ldots,\lambda_n\neq 0$. 
\end{lemma}
\begin{proof}
    Let $a_w$ be the coefficient of $x^w:=x_0^{w_0}\cdots x_n^{w_n}$ for $w=(w_0,\cdots, w_n)$ and $|w|:= w_0+\cdots+w_n=d$. Now, fix  $w$
    such that $|w|=d$ and $w\neq d\,e_i$ for some $i$. Here, $e_i$ is the $i$--th standard vector $(0,\ldots,0,1,0,\ldots,0)$. In particular, there exists $i\neq j$ with $w_i,w_j\geq 1$. We deduce that  $a_w$ is the coefficient of the monomial $x^{w-e_i-e_j}$ of $\frac{\partial^2 G}{\partial x_ix_j}$. Since $G$ and $F_{n+1}$ have the same Hessian variety, from equation \eqref{eq:proj sym} we get that $\frac{\partial^2 G}{\partial x_ix_j}=0$. Therefore, $a_w=0$ for  $w\not\in \{ de_0,\ldots,de_n\}$, and $G = \lambda_0x_0^d+\cdots+\lambda_nx_n^d$ for $\lambda_i\in\K$.

    Now assume that $\lambda_n=0$. Then, the Hessian variety of $G$ would be a hypersurface in a projective space of dimension $n-1$. This is a contradiction since the Hessian variety of $F_{n+1}$ has dimension $n-1$. We conclude that $\lambda_i\neq 0$ for every $i$.
\end{proof}

\para 

Let $\Gamma$ be the linear subspace spanned by $x_0^d,\ldots,x_n^d$ in $S^dV$. We identify a point $\lambda:=[\lambda_0,\ldots,\lambda_n]\in\P(\Gamma)$ with the polynomial $F_\lambda:= \lambda_0x_0^d+\cdots+\lambda_nx_n^d$. Using this notation $F_{n+1}$ equals $F_{\mathbbm{1}}$, where $\mathbbm{1}:=[1,\ldots,1]$. We denote the $n$--dimensional torus of $\P(\Gamma)$ by $\T$.
From Lemma \ref{lemma:sym fibers}, we deduce that if $G$ has the same Hessian variety as $F_{\mathbbm{1}}$, then $G$ must lie in the torus $\T$.
Given $F_\lambda\in\P(\Gamma)$,
 we consider the rational map 
\[\begin{array}{cccc}
 h_\lambda:&\P^n& \dashrightarrow &\P(U)\\
  &[x_0,\ldots,x_n]&\longmapsto&[\lambda_0x_0^{d-2},\ldots,\lambda_nx_n^{d-2}]
  \end{array},
\]

Note that for $F_\lambda\in\T$, $h_\lambda$ is well-defined everywhere and it has degree $(d-2)^n$. The image of $\bbV(F_\lambda)$ is a hypersurface defined by an homogeneous polynomial $\tilde{F}_\lambda$. We denote the degree of $\tilde{F}_\lambda$ by $\tilde{d}$. 

\para 

\begin{lemma}\label{lemma: deg F tilde}
    For $F_\lambda\in\T$, the degree $\tilde{d}$ of the polynomial $\tilde{F}_{\lambda}$ equals $d(d-2)^{n-1}$ if $d$ is odd, and $ d(d-2)^{n-1}/2^n $ if $d$ is even.
\end{lemma}
\begin{proof}
    Since $h_F$ is finite, we get that the  multiplication of degree of $\tilde{F}_\lambda$ and the degree of the restriction of $h_\lambda$ to $\bbV(F_\lambda)$ equals $d(d-2)^{n-1}$.
    By Lemma \ref{lemma:n=1 2}, it is enough to prove the statement for $\tilde{F}_\mathbbm{1}$. Hence, it is sufficient to check that 
    \[
    \mathrm{deg}\, h_\mathbbm{1} |_{\bbV(F_\mathbbm{1})}=\left\{
    \begin{array}{ll}
        1 & \text{if } d \text{ is odd,} \\
       2^n  &  \text{if } d \text{ is even.} 
    \end{array}
    \right.
    \]
    Let $p$ be a generic point in $\bbV(F_\mathbbm{1})$ and let $q\in \bbV(F_\mathbbm{1})$  be such that $h_\mathbbm{1}(p)=h_\mathbbm{1}(q)$. Then, there exist $\mu\neq 0$ and $\xi_0,\ldots,\xi_n$ $(d-2)$--roots of the unit, such that $q_i = \mu \xi_ip_i$ for every $i\in \{0,\ldots,n\}$. Since $\sum q_i^d =0$, we deduce that 
    \[
    \xi_0^2p_0^d+\cdots +\xi_n^2p_n^d=p_0^d+\cdots+p_n^d=0.
    \]
    In particular, we have that $(\xi_0^2-1)p_0^d+\cdots +(\xi_n^2-1)p_n^d=0$. Assume that $n\geq 2$. Since this equality holds for $p\in\bbV(F_\mathbbm{1})$ generic and $\dim \bbV(F_\mathbbm{1})\geq 1$, we get that  $\xi_i^2=1$ for every $i$. Since $\xi_i^{d-2} = 1$, we deduce that $\xi_i = \pm 1$ if $d$ is even or $\xi_i = 1$ if $d$ is odd. 
    For $n=1$, we have that $p_0^d = -p_1^d$, $p_0,p_1\neq 0$ and $\xi_0^2p_0^d = -\xi_1^2p_1^d$. This implies that $\xi_0^2=\xi_1^2$. Since we can assume $\xi_0 = 1$, we deduce that $\xi_0^2=\xi_1^2 =1$.
    Hence, we conclude that the fiber $h_\mathbbm{1}^{-1}(h_\mathbbm{1}(p))$ is the point $p$ if $d $ is odd or the $2^n$ points
    of the form $[\pm p_0,\ldots, \pm p_n]$ if $d$ is even.
\end{proof}

\para 

Lemma \ref{lemma: deg F tilde} allows us to define the rational map 
\[
\tilde{H}_{d,n}:\P(\Gamma)\dashrightarrow\P(S^{\tilde{d}} U)
\]
sending $F_\lambda$ to $\tilde{F}_\lambda$. Note that $\tilde{H}_{d,n}$ is well defined in the torus $\T$. In general, using Lemma \ref{lemma:n=1 2}, and the fact that the $\mathrm{PGL}(n+1)$ orbit of $F_k$ is dense in $\sigma_{k}(V^{d,n})$ for $k\leq n+1$, we get that the Hessian variety of a generic element in $\sigma_{k}(V^{d,n})$ is a hypersurface of degree $\tilde{d}$ in a $k$--dimensional subspace of $\P(S^2V)$. In particular, the target space of $H_{d,n,k}$ is the projectivization of the $\tilde{d}$--th symmetric power of the universal bundle of $\mathrm{Gr}(k+1,S^2V)$ and $\tilde{H}_{d,k}$ is the restriction of $H_{d,n,k}$ to $\P(\Gamma)$. Hence, to study the  fibers of the map $H_{d,n,k}$ it is enough to focus on $\tilde{H}_{d,n}$.

\para 

\begin{proposition}\label{prop:Htilde}
    For $d\geq 3$, the restriction of  $\tilde{H}_{d,n}$ to the torus $\T$ is finite and  \'etale. For $d$ odd it has degree $2^n$ and for $d$ even it is an isomorphism.
\end{proposition}
\begin{proof}
   We first compute the set theoretical fiber of $\tilde{H}_{d,n}(F_\mathbbm{1})$. Let $F_\lambda\in\P(\Gamma)$ such that $\tilde{F}_\lambda = \tilde{F}_\mathbbm{1} $. Consider a point $p=[p_0,p_1,0,\ldots,0]$, with $p_0p_1\neq 0$, in the hypersurface defined by $F_\lambda$. In other words, $\lambda_0p_0^d+\lambda_1p_1^d = 0$. 
  Since $\tilde{F}_\mathbbm{1}=\tilde{F}_\lambda$, there exists $q=[q_0,\ldots,q_n]\in\bbV(F_\mathbbm{1})$ such that 
   $\phi_\lambda(p)= \phi_\mathbbm{1}(q)$.
   In other words, there exists a solution to the system of equations equations  
   \begin{equation}\label{eq:lambdap}
   \begin{array}{lcl}
   \lambda_0p_0^d+\lambda_1p_1^d&=&0,\\
   q_0^d+q_1^d &=&0,\\
   \lambda_0p_0^{d-2}q_1^{d-2}-\lambda_1p_1^{d-2}q_0^{d-2}&=&0,\\
   q_i^{d-2}& = & 0\,\,\,\, \text{ for } i=2,\ldots, n.
   \end{array}
   \end{equation}
   Since $\lambda_i,p_i,q_i\neq 0$, for $i=0,1$, we have that 
   \[
   \lambda_1\left(\frac{p_1q_0}{p_0q_1}\right)^{d-2}=\lambda_0=-\lambda_1\left(\frac{p_1}{p_0}\right)^d.
   \]
   Multiplying by $(q_0/q_1)^2$ the previous equation we get that 
   \begin{equation}\label{eq:pq}
   p_0^2q_1^2=p_1^2q_0^2.
   \end{equation}
   Assume first that $d$ is even, i.e., $d = 2a$ for $a\in \N$. Using equations \eqref{eq:lambdap} and \eqref{eq:pq} we get that 
   \[
   \lambda_0 (p_0q_1)^{d-2}=(p_1q_0)^{d-2}\lambda_1=
   (p_1^2q_0^2)^{a-1}\lambda_1 = (p_0^2q_1^2)^{a-1}\lambda_1 =  (p_0q_1)^{d-2}\lambda_1.
   \]
   In particular, we deduce that $\lambda_0 = \lambda_1$. Applying this argument to points in $\bbV(F_\lambda)$ of the form $[p_0,0,\ldots,0,p_i,0,\ldots,0]$, we deduce that $\lambda_0=\lambda_i$, and we conclude that $F_\lambda = F_n$. 

   Assume now that $d$ is odd. Applying the same reasoning as in the even case to the equation 
   \[
   (\lambda_0p_0^{d-2}q_1^{d-2})^2=(\lambda_1p_1^{d-2}q_0^{d-2})^2,
   \]
   we deduce that $\lambda_0^2 = \lambda_1^2= \cdots=\lambda_n^2$. Moreover, one can check that, for such a choice of $\lambda$, $\tilde{F}_\lambda = \tilde{F}_\mathbbm{1}$. In particular, we conclude that the reduced structure of the fiber $\tilde{H}_{d,n}^{-1}(\tilde{F}_\mathbbm{1})$ is the point $F_\mathbbm{1}$ if $d$ is even or the $2^n$ points of the form $\sum \pm x_i^d$ if $d$ is odd. By the Generic Smoothness Theorem (see \cite{Hart} Corollary 10.7), we get that a generic fiber of $\tilde{H}_{d,n}$
   is reduced. By Proposition \ref{prop:isofibers}, all the fibers of the restriction of  $\tilde{H}_{d,n}$ to $\T$ are isomorphic via the torus action and we conclude that the schematic structure of the fibers is reduced. Moreover, since   the fibers of $\tilde{H}_{d,n}|_\T$ are all reduced of the same degree, we deduce that $\tilde{H}_{d,n}|_\T$ is \'etale. In particular, for $d$ even it is an isomorphism.
   \end{proof}
Recall that for $k\leq n+1$,
$\mathcal{U}_k\subseteq \sigma_k(V^{d,n})$ is the nonempty open subset of polynomials of the form $l_1^d+\cdots+l_{k}^d$ such that $l_1,\ldots,l_{k}\in S^1V$ are linearly independent. 
   As a consequence of Proposition \ref{prop:Htilde}, we derive the following theorem. 

   \para 

   \begin{theorem}\label{theo:sym hessian}
       For $d\geq 3$ and $k\leq n+1$, the restriction of $H_{d,n,k}$ to $\mathcal{U}_k$ is finite and \'etale. For $d$ odd, it has degree $2^{k-1}$. For $d$ even, $H_{d,n,k}|_{\mathcal{U}_k}$ is an isomorphism.
   \end{theorem}
   \begin{proof}
       By Proposition \ref{prop:isofibers}, the fibers of the restriction of $H_{d,n,k}$ to $\mathcal{U}_k$ are all isomorphic to the fiber of $H_{d,n,k}(F_k)$ via the action of $\mathrm{PGL}(n+1)$. In particular, by Lemma \ref{lemma:sym fibers}, these fibers are isomorphic to the fibers of the restriction of $\tilde{H}_{d,k}$ to the torus $\T$.       
       Then, the proof follows from Proposition \ref{prop:Htilde}.
   \end{proof}

   A first consequence of Theorem \ref{theo:sym hessian} is the case of the Hessian correspondence for $n=1$ and $d=3$. In this setting, a generic polynomial in $\P(S^3\mathbb{C}^2)$ has rank $2$, and hence,  $H_{3,1,2}$ and $H_{3,1}$ coincide. In particular, we derive the following result.

   \para 

\begin{corollary}\label{cor:n=1}
The restriction of $H_{3,1}$ to the open subset of cubic binary forms with distinct roots is \'etale of degree $2$.
\end{corollary}

\section{Hessian correspondence of degree $3$}\label{sec:d=3}

In this section we focus on the Hessian correspondence $H_{d,n}$ for degree $3$ hypersurfaces.
For $F\in\P(S^3V)$, its second derivatives are linear forms. Therefore, the Hessian map is a linear map. Since for $x_0^3+\cdots+x_n^3$ the Hessian map is an embedding, we deduce that  for a $F$ generic, $h_F$ is a linear embedding. 
In particular, we deduce that for $F,G\in\P(\S)$ generic  such that $H_{3,n}(F) = H_{3,n}(G)$, it holds that $\bbV(F)$ and $\bbV(G)$ are isomorphic as varieties. 
The next proposition analyzes the relation between two generic hypersurfaces of degree three with same Hessian variety.  

\para

\begin{proposition}\label{prop:3.1.uptoiso}
Let $F\in \P(\S)$ be generic and $G\in\P(\S)$  be such that $H_{3,n}(F) = H_{3,n}(G)$. Then, there exists $g\in\pgn$ such that $F\circ g = G$.
\end{proposition}
\begin{proof}
$H_{3,n}(F)$ is contained in $h_F(\P^n)\cap h_G(\P^n)$. Assume that $h_F(\P^n)\neq h_G(\P^n)$. Since $\dim H_{3,n}(F) = n-1$, we deduce that $H_{3,n}(F) = h_F(\P^n)\cap h_G(\P^n)$. Using that $h_F$ is a linear embedding, one has that $F$ is the power of a linear form. Since $F$ is generic, we get that $h_F(\P^n)= h_G(\P^n)$. Let $g = h_F^{-1}\circ h_G$.  Now, the proof follows from the fact that $g$ is an automorphism of $\P^n$ that maps $\bbV(G)$ to $\bbV(F)$.
\end{proof}

\para 

\begin{remark}\label{remark:lin d=3}
    As a consequence of the proof of Proposition \ref{prop:3.1.uptoiso} we deduce that for generic $F\in \P(S^3V)$, $h_F(\P^n)$ is the unique $n$ dimensional linear subspace containing the Hessian variety. In particular, $h_F(\P^n)$ is the smallest linear subspace containing $h_{3,n}(F)$.
\end{remark}

\para 

We are interested in giving a better description of the fibers of $H_{3,n}$. For instance, we saw in Corollary \ref{cor:n=1} that for $d=3$ and $n=1$ the Hessian correspondence has generically degree $2$. However, to derive this result we use that a generic cubic binary form can be expressed as a sum of two cubes. In general, it is not true that a cubic form in $n+1$ variables has Waring rank $n+1$. As a consequence we distinguish two cases in our study, $n=1$ and $n\geq 1$.

\subsection{Case $n=1$}\label{sec 3,1}

In this subsection we fix $d=3$ and $n=1$. In Corollary \ref{cor:n=1} we showed that the restriction of $H_{3,1}$ to the open subset of polynomials with distinct roots is a degree $2$ \'etale map. The aim of this section is to find the involution preserving the fibers of $H_{3,1}$, provide a recovery algorithm for $H_{3,1}$ and compute the image of $H_{3,1}$. 

\para 

For $F\in\P(S^3V)$, $\bbV(F)$ consists of three points and $h_F$ is a linear map. The locus where $h_F$ is not a linear embedding coincides with the twisted cubic $C$ in $\P(S^3V)$ of cubes of linear forms. Indeed, let $F\in C$. After a change of coordinates we can assume $F=x_0^3$ for which $h_F$ is not a linear embedding. Now, assume that $F$ is not in $C$. We can assume that $F$ is $x_0^2x_1$ or $x_0x_1(x_0+x_1)$. One can check that $h_F$ is a linear embedding in both cases. 
 Hence, for $F\not \in C$, $h_F$ is an embedding and $H_{3,1}(F)$ consists of $3$ points in $\P^2$ counted with multiplicity. 
Based on this, 
we consider $H_{3,1}$ as  the rational map $ H_{3,1}:\P(\S)\dashrightarrow\mathrm{Sym}^3\P^2$ sending $F$ to its Hessian variety.
In Corollary \ref{cor:n=1} we showed that the restriction of $H_{3,1}$ to the open subset of polynomials with distinct roots is a degree $2$ \'etale map. Moreover,
in Section \ref{sec:sym} we saw that $x_0^3+x_1^3$ and $x_0^3-x_1^3$ have the same Hessian variety. However, for our purposes, it will be useful to focus on a different example.

\para 

\begin{example}\label{ex:n=1 1}
Consider the polynomials 
 $F=x_0x_1(x_0-x_1)$ and $G=(x_0-2x_1)(2x_0-x_1)(x_0+x_1).$
Then, one can check that the Hessian varieties of $F$ and $G$ are both equal to 
 $\{[1,-1,0],[0,-1,1],[1,0,-1]\}$. In particular, there is no other polynomial in $\P(S^3V)$ with the same Hessian variety as $F$ and $G$. 
\end{example}

\para

  Let $F,G\in\P(\S)$ with  distinct roots. Then, since $\pgl$ is $3$--transitive in $\P^1$, the set of elements $g\in\pgl$ such that $g\cdot F = G$  consists in $6$ elements corresponding to the permutation of the roots of $F$. If $F$ and $G$ have $2$ distinct roots, then the above set is isomorphic to $6$ copies of $\P^1$.  In this situation, for $F\in\P(\S)\setminus C$, we introduce the subgroup
\[
\Sigma_F:= \{g\in\pgl: \rho(g^{\mathrm{t}})^{\mathrm{t}}( H_{3,1}(F))=H_{3,1}(F)\}. 
\]
As a consequence of Lemma \ref{lemma:n=1 2} we deduce the following proposition.

\para 

\begin{proposition}\label{prop:bij n=1}
  Let $F\in\P(\S)\setminus C$ . Then, there is a surjection
  \begin{equation}\label{eq:surj n=1}
  \begin{array}{ccc}
  \Sigma_F&\rightarrow&H_{3,1}^{-1}(H_{3,1}(F))\\
  g&\mapsto& g\cdot F.
  \end{array}
  \end{equation}
  Moreover, the preimage of a polynomial $G$ consists of the elements of $\pgl$ that map the three roots of $G$ to the three roots of $F$.
\end{proposition}
\begin{proof}
Let $g\in\pgl$ be such that $\rho(g^{\mathrm{t}})^{\mathrm{t}}$ preserves $H_{3,1}(F)$ and let  $G=g\cdot F$. By Lemma \ref{lemma:n=1 2}, we deduce that 
$
H_{3,1}(G) = \rho(g^{\mathrm{t}})^{\mathrm{t}}(H_{3,1}(F)) =H_{3,1}(F),
$
and we conclude that 
$F$ and $ G$ have the same Hessian variety. Hence, the map \eqref{eq:surj n=1} is well defined. Now, let $G\in \P(S^3V)$ with same Hessian variety as $F$. Since the action of $\pgl$ on $\P^1$ is $3$--transitive, there exists $g\in\pgl$ such that $g \cdot F = G$ and by Lemma \ref{lemma:n=1 2}, $\rho(g^{\mathrm{t}})^{\mathrm{t}}(H_{3,1}(F) ) = H_{3,1}(G) = H_{3,1}(F)$. Hence, we conclude that $g$ lies in $\Sigma_F$ and the map  \eqref{eq:surj n=1} is surjective.
\end{proof}

\para 

\begin{example}\label{ex:n=1 2}
Let $F$ and $G$ be the polynomials in Example \ref{ex:n=1 1}. Then, $\Sigma_F$ consists in the matrices
\[\begin{array}{c}
\left(\begin{array}{cc}1 &0\\0&1\end{array}\right),
\left(\begin{array}{cc}0 &1\\1&0\end{array}\right),
\left(\begin{array}{cc}1 &0\\-1&-1\end{array}\right),
\left(\begin{array}{cc}1 &1\\0&-1\end{array}\right),
\left(\begin{array}{cc}0 &-1\\1&1\end{array}\right),
\left(\begin{array}{cc}1 &1\\-1&0\end{array}\right),\\
\noalign{\vspace*{1mm}}
\left(\begin{array}{cc}1 &2\\-2&-1\end{array}\right),
\left(\begin{array}{cc}-2 &-1\\1&2\end{array}\right),
\left(\begin{array}{cc}1 &2\\1&-1\end{array}\right),
\left(\begin{array}{cc}-1 &1\\2&1\end{array}\right),
\left(\begin{array}{cc}2 &1\\-1&1\end{array}\right),
\left(\begin{array}{cc}-1 &1\\-1&-2\end{array}\right).\\
\end{array}
\]
The first six matrices are the automorphisms of $\P^1$ that preserve $\bbV(F)$. The last six matrices are the automorphisms $g$ of $\P^1$ that map  $g\cdot F = G$. In particular, the surjection in \eqref{eq:surj n=1} sends the first six matrices to $F$ and the last six matrices to $G$.
\end{example}

\para 

In Corollary \ref{cor:n=1}, we restrict the study of $H_{3,1}$ to the open subset of polynomials with Waring rank $2$, which coincide with the polynomials with distinct roots. Now, we focus on the case where the polynomial has a double root. 
For  $F=x_{0}^{2} x_{1}$,  $H_{3,1}(F) = \{2[1,0,0],[0,1,0]\}$ where $2[1,0,0]$ denotes the point $[1,0,0]$ with multiplicity $2$. Then, $\Sigma_F$ consists of $g\in\pgl$ such that $\rho(g^{\mathrm{t}})^\t([1,0,0]) = [1,0,0]$ and $\rho(g^{\mathrm{t}})^\t([0,1,0])=[0,1,0]$. One can check that 
\[
\Sigma_F = \left\{\left(
\begin{array}{cc}
a_0&0\\0&a_1
\end{array}
\right):[a_0,a_1]\in\P^1\right\}\simeq \P^1.
\]
Since for every $g\in \Sigma_F$, $g\cdot F = F$, we get that $H_{3,1}^{-1}(H_{3,1}(F))=\{F\}$. By a similar argument as in the proof of Proposition \ref{cor:n=1}, we get that for $F\in\P(S^3V)$, with two distinct roots,  $H_{3,1}^{-1}(H_{3,1}(F))=\{F\}$ set-theoretically.

\para 

\begin{remark}\label{rem:fiberH31} \textsf{(Recovery of the fibers of $H_{3,1}$)} The above study provides an effective method for computing preimages through $H_{3,1}$. Assume that we are given $\{p_1,p_2,p_3\}$ in the image of $H_{3,1}$ where $p_1,p_2,p_3$ are distinct. Consider the ideal $I$ defined by the $2\times 2$ minors of the matrices 
\[
\begin{array}{c}
\left(
\begin{array}{ccc}
a(a-2c)&ab-ad-bc & b(b-2d) \\ p_{1,0}&p_{1,1}&p_{1,2}
\end{array}\right),\,\,
\left(
\begin{array}{ccc}
 c(c-2a)&cd-ad-bc&d(d-2b)\\ p_{2,0}&p_{2,1}&p_{2,2}
\end{array}\right), 
\\
\noalign{\vspace*{2mm}}
\text{ and }
\left(
\begin{array}{ccc}
a^2-c^2&ab-cd&b^2-d^2\\ p_{3,0}&p_{3,1}&p_{3,2}
\end{array}\right),
\end{array}
\]
where $p_i=[p_{i,0},p_{i,1},p_{i,2}]$.
Then, $\bbV(I)$ is the subvariety of $\pgl$ consisting of the elements $g\in\pgl$ such that $\rho(g^{\mathrm{t}})^\t( [1,-1,0]) = p_1$, $\rho(g^{\mathrm{t}})^\t([0,-1,1]) = p_2$, and $\rho(g^{\mathrm{t}})^\t([1,0,1] )= p_3$. Then, $\bbV(I)$ consists of two involutions $g_1$ and $g_2$. In particular, these two involutions can be explicitly computed, since the polynomials
\[\begin{array}{c}
p_{2,0}(p_{1,1}-p_{3,1})a^2+2p_{3,0}(p_{2,1}-p_{1,1})ac+p_{1,0}(p_{3,1}-p_{2,1})c^2,\\
p_{2,2}(p_{3,1}-p_{1,1})b^2+2p_{3,2}(p_{1,1}-p_{2,1})bd+p_{1,2}(p_{2,1}-p_{3,1})d^2
\end{array}
\]
lie in $I$. Solving these equations lead to $4$ possible solutions. Discarding the two solutions that do not vanish at $I$, we get $g_1$ and $g_2$.
Then, the two polynomials in $H_{3,1}^{-1}(\{p_1,p_2,p_3\})$ are 
$
g_1 \cdot x_0 x_1(x_0-x_1)$ and $g_2 \cdot x_0 x_1(x_0-x_1)
$.
\end{remark}

\para

As a consequence of Corollary \ref{cor:n=1}, there exists a rational involution in $\varphi$ in $\P(\S)$ that preserves the fibers of $H_{3,1}$. Let $F,G\in\P(\S)$ with distinct roots such that $H_{3,1}(F)=H_{3,1}(G)$. Let $h_1\in\pgl$ such that 
$h_1\cdot F = x_0x_1(x_0-x_1)$ and consider the automorphism 
$$
h_2 = \begin{pmatrix}
1&2\\-2&-1
\end{pmatrix}
$$
appearing in Example \ref{ex:n=1 2}. Then, for $g =h_1^{-1}h_2h_1 $ we get that $g\cdot F = G$.
This defines a rational map 
$
\psi:\P(\S)\dashrightarrow\P(\mathrm{Mat}(2,2))
$
sending $F$ to $h_1^{-1}h_2h_1 $ given by 
\begin{equation}\label{eq:mat involution}
[a_0,a_1,a_2,a_3]\mapsto
\begin{pmatrix}
9a_3a_0-a_1a_2 & 
2a_1^2-6a_2a_0\\
6a_3a_1-2a_2^2&a_2a_1-9a_3a_0
\end{pmatrix},
\end{equation}
where $a_0,\ldots,a_3$ are the coordinates of $\P(\S)$. Note that the determinant of the matrix in \eqref{eq:mat involution} is the equation defining the secant variety to $C$.
 Then, the desired rational involution
 $
\iota:\P(\S)\dashrightarrow\P(\S)
$
sends $F$ to $\psi(F)\cdot F$. In coordinates, $\iota$ is given by four  homogeneous polynomial of degree $7$. One can check that the determinant of the matrix in \eqref{eq:mat involution} divides these four polynomials. We conclude that $\iota$ is given by three polynomials of degree four.
A computation using the software \textsc{Macaulay2}  \cite{M2} shows that the four polynomials defining the involution $\iota$ are 
\[ \begin{array}{c}
-2a_1^3+9a_0a_1a_2-27a_0^2a_3, \text{ }
3(-a_1^2a_2+6a_0a_2^2-9a_0a_1a_3), \text{ }
3(a_1a_2^2-6a_1^2a_3+9a_0a_2a_3)\\ \text{ and }
2a_2^3-9a_1a_2a_3+27a_0a_3^2.
\end{array}
\]
Moreover, the radical of the ideal generated by these polynomials is the ideal of the twisted cubic $C$.

\para

We finish this section with the computation of the image of $ H_{3,1}$. To do so, we embed $\mathrm{Sym}^3\P^2$ in $\P^9$ through the global sections of $\mathcal{O}_{(\P^2)^3}(1,1,1)^{S_3}$.

\para 

\begin{proposition}
The image of $H_{3,1}$ is a subvariety of $\P^9$ of dimension $3$, degree $10$ and its ideal is generated by $3$ linear forms and $10$  cubic forms.
\end{proposition}
\begin{proof}
    This result is obtained by a direct computation using the software \textsc{Macaulay2}  \cite{M2}.
\end{proof}

\subsection{Variety of $k$--gradients}\label{sec:var k planes}

 Now, we introduce the algebraic objects we need for studying the map $H_{3,n}$ for $n\geq 2$. 
For $F\in \P(\S)$, we consider the vector space 
spanned by it first derivatives $\langle \nabla F\rangle:=\langle \frac{\partial F}{\partial x_0},\ldots,\frac{\partial F}{\partial x_n}\rangle\subseteq S^2V$.  As commented in the introduction, the strategy for studying the map $H_{3,n}$ for $n\geq 2$ is to look at $h_F(\P^n)$  instead of $H_{3,n}(F)$. 

\para

\begin{lemma}\label{lemma:grad hess}
 For $F\in\P(\S)$, $\P(\langle \nabla F\rangle)=h_F(\P^n)$.
\end{lemma}
\begin{proof}
 By the Euler formula, the $(a,b)$-th entry of $2\frac{\partial F}{\partial x_i}$, as a symmetric matrix in $\P(\SS)$, is the third derivative $\frac{\partial^3 F}{\partial x_i\partial x_a \partial x_b}$. On the other hand,
let $\{e_0,\ldots,e_n\}$ be the canonical basis of $\P^n$. Then, $h_F(\P^n)$ is the span of $\{h_F(e_0),\ldots,h_F(e_n)\}$. The $(j,k)$-th entry of $h_F(e_i)$ is the coefficient of $x_i$ in $\frac{\partial^2F}{\partial x_i\partial x_j}$, which is  $\frac{\partial^3 F}{\partial x_i\partial x_j \partial x_k}$. Hence, the symmetric matrices $2F_i$ and $h_F(e_i)$ coincide, and $\P(\langle \nabla F\rangle)=h_F(\P^n)$.
\end{proof}

\para

Using this lemma, we  consider the morphism $\alpha_n$ introduced in \eqref{eq:alpha first} as the morphism mapping $F$ to $\P(\langle\nabla F\rangle)$. 
For $0\leq k\leq n$, let $\sub_k$ be the variety of polynomials $F\in\P(S^3V)$ such that $\dim\P(\langle\nabla F\rangle)\leq k$.
The ideal of $\sub_k$ is described by the $(k+2)\times( k+2)$ minors of the catalecticant matrix $\mathrm{Cat}(2,1:n+1)$ (see \cite{cat1}). In \cite{IV} Lemma 1.22, the following  description of $\sub_k$ is given:
\begin{equation}\label{eq:sub def2}
\sub_k= \{F\in \P(\S): \exists U\in\mathrm{Gr}(k+1,V) \text{ s.t. } F\in\P(S^3U)\}.
\end{equation}
Moreover, by \cite{IV} Proposition 1.23, we get that $\sub_k$ is an irreducible variety of dimension 
\begin{equation}\label{eq:dim sub}
(k+1)(n-k)+\binom{3+k}{3}-1.
\end{equation}

We consider the rational map $\alpha_k:\sub_k\dashrightarrow \Gr(k,\P(\SS))$ sending $F$ to $\grad$.
The base locus of $\alpha_k$ is $\sub_{k-1}$. 
We define the variety of $k$--gradients $\phi_k$ as
\[
\phi_k:= \{E\in\Gr(k,\P(\SS)):\exists F\in\P(\S)\text{ s.t. } \P(\langle\nabla F\rangle )\subseteq E\}.
\] 
Similarly, we define the variety of full dimensional $k$--gradients $\mathcal{Z}_k$ as the closure of the image of $\alpha_k$.  Therefore, from the irreducibility of $\sub_k$ we deduce that $\cZ_k$ is irreducible for every $k$. 

\para 

From the definition of $\phi_k$, it is not clear whether $\phi_k$ is a projective algebraic variety.
To check this, we decompose  $S^2V\otimes V$ as $\mathrm{Gl}(V)$--module. Applying the Pieri formula (see \cite{Ten}), we get that $
S^2V\otimes V = S^3V\oplus \mathbb{S}_{(2,1)}V,
$
and we have a short exact sequence 
\begin{equation}\label{eq:ses}
\begin{tikzcd}
0 \arrow[r] & S^3V \arrow[r] & S^2V\otimes V \arrow[r, "\pi"] & {\mathbb{S}_{(2,1)}} V \arrow[r] & 0
\end{tikzcd}.
\end{equation}

The following result was provided by Fulvio Gesmundo in a personal communication.

\para 

\begin{lemma}\label{lemma:ker E}
For $k\leq n$ and $E\in\mathrm{Gr}(k+1,\SS)$ consider the restriction
$
\pi|_E:E\otimes V\rightarrow\mathbb{S}_{(2,1)}
$, where $\pi$ is the surjection in \eqref{eq:ses}. Then, 
$
\mathrm{Ker}(\pi|_E)=\{F\in\S:\langle\nabla F\rangle\subseteq E \}.
$
\end{lemma}
\begin{proof}
Let 
$
F=\sum_{i=0}^n f_i\otimes x_i\in\mathrm{Ker}(\pi|_E)
$
where $f_i\in E$. Using the Euler formula we get that
$
F=\sum_{i=0}^n \frac{\partial F}{\partial x_i}\otimes x_i.
$
We  deduce that $F_i=f_i$ and hence, $\langle\nabla F\rangle\subseteq E$.

Conversely, consider  $F=\sum_{i=0}^n F_i\otimes x_i\in \S$ such that $\langle \nabla F\rangle \subseteq E$.  Using the Euler formula we get 
$
F=\sum_{i,j} (\frac{\partial^2 F}{\partial x_i\partial x_j}\otimes x_j)\otimes x_i.
$
In particular, we get that
\[
\begin{array}{ccl}
\displaystyle\pi(F) &=& \sum_{i,j} F_{i,j}\otimes (x_i\wedge x_j)\\
&=&\sum_{i<j}F_{i,j}\otimes (x_i\wedge x_j)+\sum_{j<i} F_{i,j}\otimes (x_i\wedge x_j)\\ &=&  \sum_{i<j}F_{i,j}\otimes (x_i\wedge x_j)-\sum_{j<i} F_{i,j}\otimes (x_j\wedge x_i)=0.
\end{array}
\]
\end{proof}

 From this lemma, we derive the following result
 
\para

\begin{corollary}\label{co:phi var}
For $k\leq n$, $\phi_k$  is a closed subvariety of $\Gr(k,\P(\SS))$.
\end{corollary}
\begin{proof}

Consider the  map $\beta:\Gr(k,\P(\SS))\rightarrow\P\left(\mathrm{Mat}(S^2V\otimes V, \mathbb{S}_{(2,1)})\right) $
sending a $k$--plane $\P(E)$ to the composition of the projection $S^2V\otimes V\rightarrow E\otimes V$ with $\pi:S^2V\otimes V\rightarrow \mathbb{S}_{(2,1)}$. Let $l:=\mathrm{max}\{\dim(E\otimes V),\dim(\mathbb{S}_{(2,1)})\}$. Then, the image of $\beta$ lies in the subvariety of matrices of rank at most $l$. The preimage of the subvariety of matrices of rank at most $l-1$ is $\phi_k$. Indeed, by Lemma \ref{lemma:ker E}, $E\in\mathrm{Gr}(k+1,\SS)$ lies in $\phi_k$ if and only if 
$\pi|_E$ does not have maximum rank, if and only if $\beta(E)$ has rank at most $l-1$.
\end{proof}

\para 

By construction, the image of $\alpha_k$ lies in $\phi_k$ and since $\phi_k$ is closed,  we get that 
$\cZ_k\subseteq\phi_k$. Then, for $0\leq l\leq k \leq n$ we consider the variety $F(\cZ_l,k)$ defined as
\[
F(\cZ_l,k):= \{\Gamma\in\Gr(k,\P(\SS)):\exists \Lambda\in\cZ_l\text{ s.t. } \Lambda\subseteq\Gamma
\}.
\]

\para

\begin{proposition}\label{prop:desc phi_k}
For $l\leq k\leq n$, $F(\cZ_l,k)$ is an irreducible subvariety of $\phi_k$ and 
\[
\phi_k=\bigcup_{l\leq k} F(\cZ_l,k)
\]
\end{proposition}
\begin{proof}
Consider the subvariety 
$
\Sigma:=\{(E,F)\in \cZ_l\times \Gr(k,\P(\SS)):E\subseteq F\}
$
and the  projections 
$
\pi_1:\Sigma\rightarrow \cZ_l
$ and $
\pi_2:\Sigma\rightarrow  \Gr(k,\P(\SS)).
$ 
All the fibers of $\pi_1$ are irreducible and of the same dimension. Since $\cZ_l$ is irreducible, by \cite{HR} Theorem 11.14, we deduce that $\Sigma$ is irreducible too. 
 Since the image of $\pi_2$ is $F(\cZ_l,k)$, we conclude that $F(\cZ_l,k)$ is irreducible.

For $\Gamma\in\phi_k$, there exists $F\in\P(S^3V)$ such that $\P(\langle \nabla F\rangle )\subseteq \Gamma$. Thus, denoting the dimension of $\P(\langle \nabla F\rangle)$ by $l$, we get that $\Gamma$ lies in $F(\cZ_l,k)$. On the other hand, for $l\leq k$ and $\Gamma\in F(\cZ_l,k)$ generic, there exists $F\in \sub_l$ such that $\P(\langle \nabla F\rangle)\subseteq \Gamma$. Therefore, $\Gamma\in \phi_k$, and since $\phi_k$ is closed and $F(\cZ_l,k) $  is irreducible, we conclude that $F(\cZ_l,k)\subset \phi_k$.
\end{proof}

\para 

As a consequence, we get that the irreducible components of $\phi_k$ are of the form $F(\cZ_l,k)$ for $0\leq l\leq k$.
In the cases $k=0$ and $k=1$, $\phi_k$ has an simple description.
\begin{itemize}
\item For $k =0$, the base locus of $\alpha_0$ is empty and  $\sub_0$ is  the Veronese variety $V^{3,n}$ of cubes of linear forms.
Thus,  for   $L^3\in \sub_0$, we get that $\alpha_0(L^3)= L^2 \in \P(\SS)$. Thus, $\cZ_0=\phi_0$ is the Veronese variety 
$
V^{2,n}
$.
\item For $k=1$, $\sub_1 = \sigma_2(V^{3,n})$, where $\sigma_2(V^{3,n})$ is the secant variety of $\mathcal{V}_3$ (see \cite{cat1} Theorem 3.3.). Let $F\in \sub_1\setminus\sub_0$. Then, either $F=L_1^3 +L_2^3$ or $F=L_1^2L_2$ for distinct $L_1,L_2\in\P(S^1V)$. In the first case, $\P(\langle \nabla F\rangle) = \langle L_1^2,L_2^2\rangle$ which is a chord of $\mathcal{V}_2$. In the second case, 
 $\P(\langle \nabla F\rangle )=\langle L_1^2,L_1L_2\rangle$ which is a tangent line of $V^{2,n}$. Hence, we deduce that $\mathcal{Z}_1$ is the variety of secant lines  $\sigma_2(V^{2,n})$
 \end{itemize}
 
Therefore, we derive the following result.

\para 

\begin{corollary}
In the cases $k=0,1$,
$
\phi_0=\mathcal{Z}_0=V^{2,n},$ $\mathcal{Z}_1=\sigma_2(V^{2,n}),$ and $\phi_1 = F(V^{2,n},1).
$
In particular, $\dim\cZ_0=\dim\phi_0 = n$, $\dim\cZ_1 = 2n$, and $\dim \phi_1 = n-2+\binom{n+2}{2}$.
\end{corollary}

\para

In Section \ref{sec:irred grad} we check that for $k\geq 2$ $\phi_k$  is not irreducible and we compute its irreducible components, as well as the dimension of each component. 

\subsection{Birationality of $\alpha_k$}\label{sec:alphak}

The aim of this subsection is to prove the birationality of $\alpha_k$ for $k\geq 2$. Consider the vector space $(S^2V)^{\oplus n+1}$ with coordinates $u_{i,j}^l$, where $u_{i,j}^l$ corresponds to the monomial $2x_ix_j$ if $i\neq j$ or  to the monomial $x_i^2$ if $i=j$ in the $l$--th direct summand. Let $W$ be the image of the linear map 
$$\begin{array}{cccc} \iota:& \P(S^3V) &\rightarrow& \P\left((S^2V)^{\oplus n+1}\right)\\
& F& \mapsto &\left(\frac{\partial F}{\partial x_0},\ldots,\frac{\partial F}{\partial x_n}\right).
\end{array}$$
By the Euler's formula, $\iota$ is a linear embedding and $W\simeq \P(S^3V)$. 
Using this projective subspace we can compute the fibers of $\alpha_k$. More precisely, let the linear subspace $\P(\Gamma)$ be in the image $\alpha_k$ and let $F\in \mathrm{Sub}_k$. Then, $\alpha_k(F) = \P(\Gamma)$ if and only if $\iota(F)\in W\cap \P(\Gamma^{\oplus n+1})$.
The projective subspace $W$ is described in the following lemma. 

\para 

\begin{lemma}
    The equations of $W$ are 
 \begin{equation}\label{eq:uijk}
\left\{ \begin{array}{ll}
u_{i,j}^i=u_{i,i}^j & \text{ for } i<j,\\
\noalign{\vspace*{1mm}}
u_{j,i}^i=u_{i,i}^j & \text{ for } j<i,\\
\noalign{\vspace*{1mm}}
u_{i,j}^l=u_{i,l}^j=u_{j,l}^i & \text{ for } i<j<l.
\end{array}  \right.
\end{equation}
\end{lemma}
\begin{proof}
We write $G\in\P(\S)$  as 
\begin{equation}\label{eq:G 3}
G=\displaystyle\sum_{i=0}^n a_ix_i^3+\sum_{i\neq j}^n b_{i,j} x_i^2x_j+\sum_{i<j<l}^n 2c_{i,j,l}x_ix_jx_l\,\,\,\,\,\,\text{ for } a_i,b_{i,j},c_{i,j,k}\in\mathbb{K}.
\end{equation}
Then, its first order derivatives are:
\[
\frac{\partial G}{\partial x_i}=3a_ix_i^2+\sum_{j\neq i}^n (2b_{i,j}x_ix_j+b_{j,i}x_j^2) +\sum_{j<l:j,l\neq i}^n 2c_{i,j,l} x_jx_l.
\]
The proof follows from the following equation
\[
u_{i,j}^l=\left\{
\begin{array}{ll}
   3a_l  & \text{ if } i=j=l, \\
   b_{l,j}  & \text{ if } i=l \text{ and } j\neq i, \\
    b_{l,i}  & \text{ if }j=l \text{ and } j\neq i, \\
    b_{j,l} & \text{ if } i=j\neq l, \\
    c_{i,j,l} & \text{ if } i,j,l \text{ are distinct.}
\end{array}
\right.
\]
\end{proof}

We start analyzing the case $k=2$. Consider the polynomial
$
 F= x_0x_1x_2\in \P(\S)
$. Then,  $\alpha_2(F)=\langle x_1x_2,x_0x_2,x_0x_1\rangle$  and  $F\in\sub_2$. Let $\Gamma$  be the intersection of $W$ with $\P(\langle \nabla F\rangle^{\oplus n+1})$, which is defined by equations \eqref{eq:uijk} and $u_{j,k}^{i}=0$ for $(j,k)\neq (0,1),(1,2),(0,2)$. One can check that $\Gamma$ is zero dimensional and, hence, for every $G\in \P(S^3V)$ such that $\gradg\subseteq\grad$, $G=F$.
Therefore, $\alpha^{-1}_2(\grad)=\{F\}$ and $\alpha_2$ is birational onto its image. 
The next result generalizes this example for $k\geq 3$.  

\para

\begin{lemma}\label{lemma:F k easy}
Let $k\geq 3$ and let $F=\sum_{i=0}^{k-1}x_i^2x_{i+1}  + x_k^2x_0.$ Then, $\P(W)\cap \P(\langle \nabla F\rangle^{\oplus n+1})$ has dimension zero.
\end{lemma}
\begin{proof}
From the first order derivatives of $F$ one can check that the equations of $\Gamma:=\P(W)\cap\P(\langle \nabla F\rangle^{\oplus n+1})$ are
\begin{equation}\label{eq:u and z}
\left\{\begin{array}{ll}
u_{i,i}^l=u_{i+1,i+2}^l & \text{ for } i\leq k-2,\\
\noalign{\vspace*{1mm}}
u_{k-1,k-1}^l=u_{0,k}^l\, , & \\
\noalign{\vspace*{1mm}}
u_{k,k}^l=u_{0,1}^l\, , & \\
\noalign{\vspace*{1mm}}
u_{i,j}^l=0 & \text{ for } (i,j)\not\in\{  (0,0),\ldots,(k,k),(0,1),\ldots,(k-1,k),(0,k)\}.
\end{array} \right. 
\end{equation}
together with the equations \eqref{eq:uijk}. We claim that 
the non vanishing coordinates of $\Gamma$ are 
$u_{0,0}^1,\ldots,u_{k-1,k-1}^k,u_{k,k}^0$ and 
$u_{0,1}^0,u_{1,2}^1,\ldots u_{k-1,k}^{k-1},u_{0,k}^k$.
 First of all, note that 
for $i,j,l\leq n$, $u_{i,j}^{l}=0 $ is an equation of $\Gamma$ if $i$, $j$ or $l$ is greater than $k$. Indeed, it holds for $i$ or $j$ greater than $k$. Assume that $i,j\leq k$ and $j>k$. By equation \eqref{eq:u and z}, if $i\neq j$ we get $u_{i,j}^{l}=u_{i,l}^j=0$. If $i=j$, $u_{i,i}^l = u_{i,l}^i=0$. Similarly, we claim that for $i<j<l$, $u_{i,j}^{l}=u_{i,l}^j=u_{i,j}^l=0$. The first three equalities appear in \eqref{eq:uijk}.
Since  $i+1<l$ we get that, by equation \eqref{eq:u and z} $u_{i,l}^j=0$ for $(i,l)\neq(0,k)$. Assume that $(i,l)=(0,k)$. For $j>1$, $u_{0,j}^k=0$. So we assume that $j=1$. Then, since $k\geq 3$,  $u_{1,k}^0 = 0$. Using that for $i<j<l$, $u_{i,j}^{l}=0$, one can check that  $u_{i,i}^i = 0$ for $i\leq n$. For example, for $i\leq k-2$, we get that by equation \eqref{eq:uijk}, $u_{i,i}^{i} = u_{i+1,i+2}^i=0 $. Now, for $i\neq k$ and $l\neq i+1$, $u_{i,i}^l = 0$. Indeed, it is enough to check this for $i,l\leq k$ and $l\neq i$. Assume first that $l>i+1$. Then, $u_{i,i}^l= u_{i,l}^l$, which, by equation \eqref{eq:uijk} equals to zero for $(i,l)\neq (0,k)$. In this case, $u_{0,0}^k=u_{0,1}^k = 0$. A similar argument shows that $u_{i,i}^k = 0 $ for $l<i-1$. For $l=i-1,i$ the claim follows from equation \eqref{eq:uijk} and the fact that 
 for $i<j<l$, $u_{i,j}^{l}=0$. Similarly, one can check that $u_{k,k}^l = 0$ for $l\neq k$.

Finally, the proof follows from the fact that the coordinates $u_{0,0}^1,\ldots,u_{k-1,k-1}^k,u_{k,k}^0$ and 
$u_{0,1}^0,u_{1,2}^1,\ldots u_{k-1,k}^{k-1},u_{0,k}^k$ are all equal by equations \eqref{eq:uijk} and \eqref{eq:u and z}.
\end{proof}

\begin{remark}
     Euler's formula allows us to write an homogeneous polynomial by means of its first order derivatives. In this sense, Lemma \ref{lemma:F k easy} is a generalization of Euler's formula for generic polynomials since it allows us to recover a homogeneous polynomial $F$ of degree $3$ from $\grad$.   
\end{remark}

\para 

As consequence of Lemma \ref{lemma:F k easy} we derive the following proposition.

\para 

\begin{proposition}\label{prop:alpha injec}
  For $2\leq k\leq n$, $\alpha_k$ is birational onto its image and $\dim\sub_k=\dim\cZ_k$.  
\end{proposition}
\begin{proof}
Let $F$  be as in Lemma \ref{lemma:F k easy}, and $G\in\P(S^3V)$ such that $\alpha_k(F)=\alpha_k(G)$. Then, $\iota(F)$ and $\iota(G)$ lie in $W\cap\P(\langle \nabla F\rangle^{\oplus n+1})$. By Lemma \ref{lemma:F k easy}, $\iota(F)=\iota(G)$. Since $\iota$ is injective, we deduce that $F=G$. So, since $F\in\sub_k$, we deduce that $\{F\}=\alpha_k^{-1}(\alpha_k(F))$. In particular,  we conclude that $\alpha_k$ is birational onto its image and that $\dim\sub_k=\dim\cZ_k$.
\end{proof}

\subsection{Case  $n\geq 2$}\label{sec:h3n>1}

In this section we prove that $H_{3,n}$ is birational onto the image for $n\geq 2$ and we provide an effective algorithm for computing $F$ from 
its Hessian variety.

\para

\begin{theorem}\label{theo:inj H3 1}
    For $n\geq 2$, $H_{3,n}$ is birational onto its image.
\end{theorem}
\begin{proof}
Let $F$ be a generic polynomial in  $\P(\S)$. As mentioned in Remark \ref{remark:lin d=3}, $h_F(\P^n)$ is the unique $n$--dimensional projective subspace containing $H_{3,n}(F)$. 
 We consider the rational map $\beta_n:\mathrm{Im} H_{3,n}\dashrightarrow \mathbb{G}\mathrm{r}(n,\P(S^2V))$ that sends $X$ to the smallest projective subspace containing $X$. By Lemma \ref{lemma:grad hess}, we get that $\alpha_n=\beta_n\circ H_{3,n}$. The proof follows  now from Proposition \ref{prop:alpha injec}.
\end{proof}

\para 

Proposition \ref{theo:inj H3 1}, and the study carried out in Subsection \ref{sec:alphak}, allow us to effectively recover $F$ from $X:=H_{3,n}(F)$. In the following algorithm, we outline the main steps for this purpose

\para 

\begin{algorithm}[H]
\caption{}\label{alg:1}

\vspace*{1mm}

\noindent \textsf{Input:}  the ideal $I$ of $X\in\mathrm{Im} H_{3,n}$.

\vspace*{1mm}

\noindent \textsf{Output:}   the unique polynomial $F\in\P(S^3V)$ such that $H_{3,n}(F)=\bbV(I)$. 

\begin{enumerate}
\item Compute the smallest projective subspace $\P(E)$ containing $\bbV(I)$ by taking the degree one part of the saturation of $I$.
\item Determine  $W\cap \P(E^{\oplus n+1})\subseteq \P\left((S^2V)^{\oplus n+1}\right)$. By Lemma \ref{lemma:F k easy}, this intersection is a point $[F_0:\cdots:F_n]\in \P\left((S^2V)^{\oplus n+1}\right)$.
\item Compute $F$  via the Euler's formula: $F = \sum x_iF_i$.
\end{enumerate}
\end{algorithm}

\para 

Let us illustrate the algorithm by an example.

\para 

\begin{example}
Let $n=2$. For $F\in \P(S^3V)$, the Hessian variety lies in $\P^5$. Consider the variety $X  = \bbV(Z_{0,0},Z_{1,1},Z_{2,2},Z_{0,1}Z_{0,2}Z_{1,2})\subset \P^5$. The smallest plane containing $X$ is 
 $\bbV(Z_{0,0},Z_{1,1},Z_{2,2})$. The intersection of $W$ with $\bbV(Z_{0,0},Z_{1,1},Z_{2,2})^{\oplus n+1}$ is given by the equations
 \[
 \begin{array}{cccc}
 u_{0,0}^0=u_{1,1}^0=u_{2,2}^0=0,
 & u_{0,1}^0=u_{0,0}^1,
 &  u_{0,1}^1=u_{1,1}^0, &\\
 
  u_{0,0}^1=u_{1,1}^1=u_{2,2}^1=0,
 & u_{0,2}^0=u_{0,0}^2,
 &  u_{0,2}^2=u_{2,2}^0,\\
  u_{0,0}^2=u_{1,1}^2=u_{2,2}^2=0,
 & u_{1,2}^1=u_{1,1}^2,
 &  u_{1,2}^2=u_{2,2}^1, &   u_{1,2}^0=u_{0,2}^1=u_{0,1}^2.
 \end{array}
 \]
 The unique solution to these equations is the point $[x_1x_2,x_0x_2,x_0x_1]$. Using the Euler's formula we get the polynomial $F=x_0x_1x_2+x_1x_0x_2+x_2x_0x_1=3x_0x_1x_2$. One can check that $H_{3,2}(F) = \bbV(Z_{0,0},Z_{1,1},Z_{2,2},Z_{0,1}Z_{0,2}Z_{1,2})$.
\end{example}

\para 

We conclude this section with some comments on the restriction of Algorithm \ref{alg:1} to $\mathrm{Sub}_k$. The proof of Theorem \ref{theo:inj H3 1} uses the birationallity of $\alpha_n$ onto its image. As a consequence, the input of Algorithm \ref{alg:1} must to be the Hessian variety of a polynomial in $\sub_n$. Nevertheless, in Proposition \ref{prop:alpha injec} we prove the birationallity of $\alpha_k$ for $k\geq 2$. Therefore, for $k\geq 2$, the restriction of $H_{3,n}$ to $\sub_k$ is birational onto its image and Algorithm \ref{alg:1} provides a method for recovering its fibers.

\subsection{Irreducible components of the variety of $k$--gradients}\label{sec:irred grad}

In this section we describe the irreducible components of the variety of $k$--gradients $\phi_k$. At the end of Section \ref{sec:var k planes}, we gave a explicit description of $\phi_0$ and $\phi_1$. In particular, we saw that both are irreducible. We will see that this is not anymore the case for $k\geq 2$. To do so, we start by generalizing Lemma \ref{lemma:F k easy} as follows. Consider the  polynomial $F = x_0x_1x_2\in\P(S^3V)$. Then, $\grad = \langle x_1x_2,x_0x_2,x_0x_1\rangle$.
Now, we consider the vector space $
\Gamma_{2,k} =\langle \nabla F\rangle + \langle x_0x_3,x_0x_4,\ldots x_0x_{k}\rangle.  
$
One can check that, the intersection $W\cap\P(\Gamma_{2,k}^{\oplus n+1})$ has dimension zero. In particular, this implies that 
for every $G\in\P(S^3V)$ such that $\langle \nabla G\rangle\subseteq \Gamma_{2,k}$, then $G=F$. The following lemma generalizes this example.

\para 

\begin{lemma}\label{lemma:gamma gen}
For every $2\leq l\leq n$, there exists $F\in\sub_l$ such that  for every $k\geq l$ there exists a vector subspace $\Gamma_{l,k}\in\mathrm{Gr}(k+1,\SS)$ with $\langle \nabla F\rangle\subseteq\Gamma_{l,k}$ satisfying that $W\cap \P(\Gamma_{l,k}^{\oplus n+1})$ is zero dimensional.
\end{lemma}
\begin{proof}
The case $l=2$ is solved at the begining of the subsection.
 Assume that $l\geq 3$ and let $F$ be as in Lemma \ref{lemma:F k easy}.
Let $\Gamma_{l,k}$ be the linear subspace defined by the equations
\begin{equation}\label{eq:gamma_lk}
\left\{\begin{array}{lr}
Z_{i,i}=Z_{i+1,i+2} & \text{for} \,\, i\leq l-2\\
Z_{l-1,l-1}=Z_{0,l} & \\
Z_{l,l}=Z_{0,1} & \\
Z_{i,j}=0 &   \text{for} \,\, (i,j)\not\in\{   (0,0),\ldots,(l,l),(0,1),\ldots,(l-1,l),\\  & (0,l),(0,l+1),\ldots,(0,k)\}.
 
\end{array}\right.
\end{equation}
where $Z_{i,j}$ are the coordinates of $\P(S^2V)$ representing the monomials $2x_ix_j $ if $i\neq j$ and $x_i^2$ otherwise. Note that $\grad\subseteq \gamma_{l,k}$ since the equations of $ \gamma_{l,k}$ are obtained by erasing from the equations of $\grad$ the equations $Z_{0,l+1}=Z_{0,l+2}=\cdots=Z_{0,k}=0 $.

We denote $W\cap \P(\Gamma_{l,k}^{\oplus n+1})$ by $W_k$. 
We claim that $W_k=W_{k+1}$ for $k\geq l$. To do so, it is enough to check that, for every $m$, the equation $u_{0,k+1}^m=0$ is among the equations of $W_k$. Assume that $k+1< m$. By \eqref{eq:uijk} we get that $u_{0,k+1}^m = u_{k+1,m}^0$ which vanishes in $W_k$. Similarly, for $m\neq 0,k+1$ we get that $u_{0,k+1}^m =0$. If $m=0$, by \eqref{eq:uijk} and \eqref{eq:u and z} get that $u_{0,k+1}^0 =u_{0,0}^{k+1}=u_{1,2}^{k+1}$. Since $k\geq 3$, we have that $u_{1,2}^{k+1}=u_{2,k+1}^1$, which vanishes in $W_k$. If $m=k+1$, we get that $u_{0,k+1}^{k+1}=u_{k+1,k+1}^0$, which vanishes in $W_k$. Therefore, we conclude that $W_l=W_k$ for every $k\geq l$. The proof follows now from Lemma \ref{lemma:F k easy}.
\end{proof}

As a consequence of the study of the maps $\alpha_k$ and Lemma \ref{lemma:gamma gen}, we can describe the irreducible components of $\phi_k$ and their dimension. By Proposition \ref{prop:desc phi_k}, the irreducible components of $\phi_k$ are of the form $F(\cZ_l,k)$ for $0\leq l \leq k$. 

\para 

\begin{lemma}
 For $0\leq l\leq k\leq n$, 
\[
\dim F(\cZ_l,k)=
 (l+1)(n-l)+\binom{3+l}{3}-1 +(k-l)\left(
\binom{n+2}{2}-1-k\right).
\]
\end{lemma}
\begin{proof}
Consider the variety 
$
\Sigma:=\{
(\Gamma,\Gamma')\in \cZ_l\times\Gr(k,\P(\SS)):\Gamma\subseteq\Gamma'
\},
$
and the projections $\pi_1:\Sigma\rightarrow \mathcal{Z}_l$ and $\pi_2:\Sigma\rightarrow\Gr(k,\P(\SS)) $. The image of $\pi_2$ is $F(\cZ_l,k)$ and, by Lemma \ref{lemma:gamma gen}, $\pi_2$ is generically injective. The fibers of $\pi_1$ are isomorphic to $\Gr(k-l-1,\binom{n+2}{2}-l-2)$. Now, the proof follows from Proposition \ref{prop:alpha injec} and equation \eqref{eq:dim sub}.
\end{proof}

Once we have computed the dimension of $F(\cZ_l,k)$, we are interested in how these dimensions relates for different $l$. 

\para

\begin{lemma}\label{lemma:dim flags}
For $n\geq 3$, the following holds:
\begin{enumerate}
    \item For $k<n$,
    $
    \dim F(\cZ_0,k)>\dim F(\cZ_2,k)>\cdots> \dim F(\cZ_{k-1},k)>\dim\cZ_k.
    $
    \item For $k=n$, $
    \dim F(\cZ_0,n)>\dim F(\cZ_2,n)>\cdots> \dim F(\cZ_{n-2},n)>\dim \cZ_n$
    and $\dim F(\cZ_{n-1},n)=\dim\cZ_n-1.$
\end{enumerate}
\end{lemma}
\begin{proof}

For $0\leq l< k$, we have that $\dim F(\cZ_l,k)-\dim F(\cZ_{l+1},k) =
\frac{n(n+1)}{2}-(k+1)-\frac{l(l+1)}{2}.$
Assume first that $k=n$. Then,
\[
\dim F(\cZ_l,n)-\dim F(\cZ_{l+1},n)= \frac{(n+1)(n-2)}{2}-\frac{(l+1)l}{2}.
\]
For $l = n-1$, the above difference is $-1$, and for $l\leq n-2$ it is non-negative.
Assume now that $k\leq n-1$. Then $l\leq n-2$ and 
\[
\dim F(\cZ_l,n)-\dim F(\cZ_{l+1},n)\geq \frac{n(n+1)}{2}-n-\frac{(n-1)(n-2)}{2} = (n-1).
\]

\end{proof}

\para 

As a consequence of this lemma,  we can describe the irreducible components of $\phi_k$.

\para 

\begin{theorem}\label{theo:irred comp}\
 \begin{enumerate}
    \item For $k<n$, $F(\cZ_0,k),F(\cZ_2,k),\ldots,F(\cZ_{k-1},k), \cZ_k$ are the irreducible components of $\phi_k$.
    \item For $k=n$, $F(\cZ_{n-1},n)\subseteq \overline{ \cZ_n\setminus \mathrm{Im}(\alpha_n)}$  and the irreducible components of $\phi_k$ are  $F(\cZ_0,k),F(\cZ_2,k),\ldots,F(\cZ_{n-2},n), \cZ_n$. 
\end{enumerate}
\end{theorem}
\begin{proof}
By Proposition \ref{prop:desc phi_k}, the irreducible components of $\phi_k$ are of the form $F(\cZ_l,k)$ for $l\leq k$. Assume first that $k<n$. Moreover, assume that $F(\cZ_l,k)$ is not an irreducible component of $\phi_k$.
Then, there exists $l'\leq k$ such that  $F(\cZ_l,k)\subset F(\cZ_{l'},k)$.  For $l\geq 2$ consider $\Gamma_{l,k}$ as in Lemma \ref{lemma:gamma gen}. If $l'<l$, then there exists $G\in\P(S^3V)$ such that $\dim \gradg \leq l'$ and  $\langle \nabla G\rangle\subseteq \Gamma_{l,k}$. By Lemma \ref{lemma:gamma gen}, we get that $\dim \gradg =l$. Therefore, we get that either  $(l',l)=(0,1)$ or $ l'>l$. 
For $l=1$, $F(\cZ_1,k)\subset F(\cZ_0,k)$, so it is not an irreducible component. For $l'>l$,
by Lemma \ref{lemma:dim flags}, we get that $\dim F(\cZ_l,k)>\dim F(\cZ_{l'},k)$. In particular, for $l\neq 1$,  $F(\cZ_l,k)$ is not contained in $ F(\cZ_{l'},k)$ and we conclude that for $l\neq 1$, $F(\cZ_l,k)$ is an irreducible component of $\phi_k$.

Assume now that $k=n$. Similarly, from Lemma \ref{lemma:dim flags} and Lemma \ref{lemma:gamma gen}, we deduce that $F(\cZ_0,k),F(\cZ_2,k),\ldots,F(\cZ_{n-2},n)$ are irreducible components of $\phi_n$ and $\dim F(\cZ_{n-1},n)=\dim\cZ_n-1$. We claim that $F(\cZ_{n-1},n)\subset \cZ_n$. Let $F$ be a generic element of $\sub_{n-1}$. By \eqref{eq:sub def2}, we can assume that $F$ is a polynomial in the variables $x_0,\ldots, x_{n-1}$. 
Let $E\in F(\cZ_{n-1},n)$ such that $\grad\subset E$. Then, $E$ is generated by $\frac{\partial F}{\partial x_0},\ldots,\frac{\partial F}{\partial x_{n-1}}$ and $\lambda x_n^2+l_1x_n+l_2$, where $l_1\in\mathrm{S}^1\langle x_0,\ldots,x_{n-1}\rangle$, $l_2\in \mathrm{S}^2\langle x_0,\ldots,x_{n-1}\rangle$, $l_2\not\in \grad$, and $\lambda\in\K$.  Since $F$ is generic, we can assume that $l_1$ and $l_2$ are generic. For $\mu\in \K$, consider the polynomial $G_\mu:= F+\mu (\frac{\lambda}{3} x_n^3+ \frac{1}{2}l_1x_n^2+l_2 x_n)$. Its second derivatives are
\[
\displaystyle
 \frac{\partial G_\mu}{\partial x_i} =    
 \frac{\partial F}{\partial x_i} +
 \frac{\mu}{2} \frac{\partial l_1}{\partial x_i}x_n^2 +  \mu\frac{\partial l_2}{\partial x_i} x_n \text{ for } i\leq n-1 \text{, and } 
\frac{\partial G_\mu}{\partial x_n} =  \mu( \lambda x_n^2+ l_1x_n+l_2) .
\]

For $\mu\neq 0$, we have that $\dim \gradg = n$. Since $l_1$ and $l_2$ are generic, the polynomials
\[
x_n^2,\frac{\partial l_2}{\partial x_{0}} x_n,\ldots,\frac{\partial l_2}{\partial x_{n-1}} x_n, \frac{\partial F}{\partial x_0},\ldots, \frac{\partial F}{\partial x_{n-1}}, \lambda x_n^2+l_1x_n +l_2
\]
are  linearly independent. We extend this set to a basis of $\SS$ and we use the following coordinates: let $y_{0,n},\ldots,y_{n-1,n},y_{n,n}$ be the coordinates of 
$ \frac{\partial l_2}{\partial x_{0}} x_n,\ldots,\frac{\partial l_2}{\partial x_{n-1}} x_n,x_n^2$, and let $y_0,\dots,y_{n-1},y_n$ be the coordinates of $ \frac{\partial F}{\partial x_{0}},\ldots,\frac{\partial F}{\partial x_{n-1}},\lambda x_n^2+l_1x_n +l_2$. Denote the rest of the coordinates by $y_i$ for $i>n$. In these coordinates, the equations of $\langle \nabla G_\mu\rangle$ are
\[  y_{0,n}=\mu y_0, \,\,
    \ldots,\,\,
    y_{n-1,n}=\mu y_{n-1}, \,\, 
    y_{n,n}=\sum_{i=0}^{n-1}\mu\frac{1}{2} \frac{\partial l_1}{\partial x_i} y_{i},  \text{ and }
     y_i = 0 \text{ for } i>n.
\]

For $\mu=0$, these equations coincide with the equations of 
$E$. Since $G_\mu\in\sub_n\setminus  \sub_{n-1}$. For $\mu\neq 0$, we get that 
$E$ lies in the closure of the image of $\alpha_n$.
\end{proof}

\section{Hessian correspondence of degree $4$ }\label{sec:d=4}

In this section we study the map $H_{d,n}$ for $d=4$. Let $U\subset\P(\Ss)$ be the open subset of polynomials $F$ whose second order derivatives are linearly independent as elements in $\SS$. Note that $U$ is non-empty since, for instance, $\sum_{i\neq j}x_i^2x_j^2$ lies in $U$. 
In particular, for $F\in U$,  its the second derivatives form a basis of $S^2V$. Therefore, we get that for $F\in U$, $h_F$ is a Veronese embedding. Hence, for generic $F,G\in U$ with same Hessian variety, one has that $\bbV(F)$ and $ \bbV(G)$ are isomorphic. However, a priori this isomorphism might not extend to an automorphism of $\P^n$.
We improve this statement in two steps. First we show that if $H_{4,n}(F)=H_{4,n}(G)$, then there exists $g\in\pgn$ such that $g\cdot F = G$. In Theorem \ref{theo:4,n} we prove that 
the only possible $g$ such that $g\cdot F = G$ is the identity. Moreover, for $n$ even we provide an effective algorithm for computing $F$ from its Hessian variety.

\para 

\subsection{Case $n=1$}\label{sec:41}

In this case, for $F\in\P(S^4V)\simeq \P^4$ generic, $H_{4,1}(F)$ consists of $4$ points in $\P^2$. These four points define a pencil of quadrics denoted by $Q_F$. We consider $H_{4,1}$ as the map $H_{4,1}:\P^4\dashrightarrow\mathbb{G}\mathrm{r}(1,\P^5)$ sending $F$ to $Q_F$,
where $\P^5\simeq \P(H^0(\P^2,\O_{\P^2}(2)))$. In the next proposition we analyze the birationality of $H_{4,1}$.

\para 

\begin{proposition}\label{prop:H41}
 For a generic $F\in\P(\Ss)$, $H_{4,1}$ is birational onto its image.
\end{proposition}
\begin{proof}
Let $a_0,\ldots,a_4$ be the coordinates of $\P^4$ and let $F=\sum a_ix_0^{4-i}x_1^i$. 
Let $b_0,\ldots,b_5$ be the coordinates of $\P^5$ corresponding to the monomials $z_{0,0}^2,z_{0,0}z_{0,1},z_{0,1}^2,z_{0,0}z_{1,1},
z_{0,1}z_{1,1},z_{1,1}^2
$.
Using the software \textsc{Macaulay2} \cite{M2}, one can check that  the ideal of $h_F(\bbV(F))$ is generated by the quadrics
\begin{equation}\label{eq:recov 41}\begin{array}{lcl}
Q_1 &:=& 3a_4z_{0,0}^2-3a_3z_{0,0}z_{0,1}+2a_2z_{0,1}^2+a_2z_{0,0}z_{1,1}-3a_1z_{0,1}z_{1,1}+3a_0z_{1,1}^2
\\
\noalign{\vspace*{1mm}}
Q_2&:=&
9a_3^2z_{0,0}^2+(-36a_2a_3+72a_1a_4)z_{0,0}z_{0,1}+(20a_2^2-144a_0a_4)z_{0,1}^2+
\\
\noalign{\vspace*{1mm}}
&&(16a_2^2-18a_1a_3)z_{0,0}z_{1,1}+(-36a_1a_2+72a_0a_3)z_{0,1}z_{1,1}+9a_1^2z_{1,1}^2.
\end{array}
\end{equation}
In particular, we see that for generic $F$, the line $H_{4,1}(F)$ is not contained in the hyperplane $\{b_2-2b_3=0\}$. Hence, we deduce that $H_{4,1}(F)$ and this hyperplane intersect in the quadric $Q_1$. Since $F$ is uniquely determined by $Q_1$ and viceversa, we conclude that $H_{4,1}$ is generically injective.
\end{proof}

\para

In Remark \ref{rem:fiberH31}, we showed how to determine the fiber of $H_{3,1}$. 
The proof of Proposition \ref{prop:H41} provides a method for computing the fibers of $H_{4,1}$.

\para

\begin{remark}\label{rem:rec H41}  \textsf{(Recovery of the fiber of $H_{4,1}$)}
Let $X\in \mathrm{Im} H_{4,1}$, and let $Q$ be the unique quadric in the intersection of  $X$ with the hyperplane defined by the equation $b_2-2b_3=0$, where $b_0,\ldots,b_5$ are the coordinates of $\P^5$; as in the proof of Proposition \ref{prop:H41}.  Then,  using equation \eqref{eq:recov 41} we can compute the unique element in $H_{4,1}^{-1}(X)$.
\end{remark}

\para

We finish this subsection dealing with the computation of the image of $H_{4,1}$. Using the Plücker coordinates, we compute this image in $\P^{14}$.

\para 

\begin{proposition}
    The image of $H_{4,1}$ in $\P^{14}$ has dimension $4$, degree $29$, and its generated by $3$ linear forms, the quadrics generating the ideal of $\Gr(1,\P^5)$ and $7$ cubic forms.
\end{proposition}
\begin{proof}
    This result is obtained by a direct computation using the software \textsc{Macaulay2} \cite{M2}.
\end{proof}

\subsection{Case $n\geq 2$}\label{4n>1}

In Section \ref{sec:d=3}, we studied $H_{3,n}$ through the linear subspace $h_F(\P^n)$. Analogously, for $d=4$, we look at the Veronese variety $h_F(\P^n)$.
In order to study $H_{4,n}$ for $n\geq 2$ we will prove that for $F\in U$,  $h_F(\P^n)$ is the unique Veronese variety containing the Hessian variety. Note that for the case $n=1$ this doesn't hold, since $H_{4,1}(F)$ is the intersection of two plane quadrics. In the rest of the subsection we fix $n\geq 2$.
The strategy to study the Veronese varieties containing the Hessian variety is to look at its first order syzygies. Let $R$ be the homogeneous coordinate ring of $\P(S^2V)$ and let $M$ be a graded $R$--module. Let 
\[M\longleftarrow F_0\overset{\delta_1}{\longleftarrow}F_1\overset{\delta_2}{\rightarrow} F_2\longleftarrow
\cdots
\]
be the minimal free resolution of $M$ where $F_i=\oplus_j R(-j)^{\beta_{i,j}}$. The integers $\beta_{i,j}$ are called the graded Betti numbers and the image of $\delta_1$ is called the module of first order sygygies of $M$. We say that the first order sygygies are linear if the image of $\delta_1$ is generated by elements of degree $1$, or equivalently, the entries of the map $\delta_1$ are linear forms (See \cite{E} for further details).

\para

\begin{lemma}\label{lemma:qua vero}
Let $X_1$ and $X_2$ be two projective subvarieties of $\P^N$ of codimension greater than $1$, defined by the same numbers of quadrics and with linear first order syzygies. Moreover, assume that $X_1$ is irreducible, nondegenerate and $X_1\cap X_2 = X_1\cap Q$ where $Q$ is a quadric hypersurface. Then, $X_1 = X_2$.
\end{lemma}
\begin{proof}
Let $I_1$ and $I_2$ be the ideals of $X_1$ and $X_2$ respectively, and let $q$ be the quadratic form defining $Q$. Then, $I_1+I_2 = \langle \alpha_1,\ldots,\alpha_m,q\rangle$, where  $\alpha_1,\ldots,\alpha_m$ are the quadrics generating $I_1$. We can assume that $I_2$ is generated by $q,\alpha_2,\ldots,\alpha_m$.  Since the first order syzygies of $X_1$ and $X_2$ are linear, there exist $l_1$ and $l_2$ linear forms such that $D(l_1)\cap X_1=\mathbb{V}(\alpha_2,\ldots,\alpha_m)$ and  $D(l_2)\cap X_2=\mathbb{V}(\alpha_2,\ldots,\alpha_m)$.  Since $X_1$ is generated by quadrics, it is nondegenerate. Therefore, $D(l_1)\cap X_1$ and $D(l_2)\cap X_1$ are non-empty. Since $X_1$ is irreducible, $D(l_1)\cap D(l_2)\cap X_1\neq \emptyset$. Similarly, $D(l_2)\cap X_2\neq \emptyset$. Hence, $D(l_1)\cap D(l_2)\cap X_1\subseteq D(l_2)\cap X_2\neq \emptyset$.
 Since $X_1$ is irreducible, we get that $X_1\subseteq X_2$, and therefore,  $I_2\subseteq I_1$. Since $I_1$ and $I_2$ are generated by the same number of linearly independent quadrics, we conclude that $X_1=X_2$. 
\end{proof}

Using this lemma, we derive the following proposition.

\para 

\begin{proposition}\label{prop:uniq vero}
 For $F\in U$, $h_F(\P^n)$ is the unique Veronese variety containing $H_{4,n}(F)$.
\end{proposition}
\begin{proof}
Assume that the Hessian variety is contained in two  Veronese varieties $V_1 = h_F(\P^n)$ and $V_2$. Then, there exists a quadratic form $q$ in the ideal of $V_2$ such that $H_{4,n}(F)\subseteq \bbV(q)\cap V_1 \varsubsetneq V_1$. Since both $\bbV(F)$ and $h_F^{-1}(\bbV(q))$ have degree $4$, we deduce that $H_{4,n}(F) = \bbV(q)\cap V_1$ and we get that  $V_1\cap V_2 = \bbV(q)\cap V_1$. By Lemma \ref{lemma:qua vero} we conclude that $V_1=V_2$.
\end{proof}

\para 

Let $F$ and $G$  be two polynomials in  $U$ with the same Hessian variety. By the previous proposition, $H_{4,n}(F)$ is contained in a unique Veronese variety. Hence, we deduce that $h_F(\P^n) = h_G(\P^n)$. Therefore, $h_G^{-1}\circ h_F$ is an automorphism of $\P^n$ that maps $\bbV(F)$ to $\bbV(G)$.  Hence, we deduce that a generic fiber of $H_{4,n}$ is contained in an orbit of the action $\pgn$ on $\P(\Ss)$. The following result provides a better description of a generic fiber of $H_{4,n}$, namely as a point.

\para 

\begin{theorem}\label{theo:4,n}
For $n\geq 2$, $H_{4,n}$ is birational onto its image.
\end{theorem}
\begin{proof}
To prove that $H_{4,n}$ is birational it is enough to check that on $U$ it is injective. Let $F,G \in U$ such that $H_{4,n}(F)=H_{4,n}(G)$. By Proposition \ref{prop:uniq vero}, $h_F(\P^n)=h_G(\P^n)$. Fixing $g=h_F^{-1}\circ h_G$ we get that $g(\bbV(G))=\bbV(F)$. Therefore,  $g^\t\cdot F=G$ and $h_G = h_F\circ g$.
As in the proof of Lemma \ref{lemma:n=1 2}, we get that
\[
\frac{\partial^2 G}{\partial x_i\partial x_j} = 
\sum_{k} 
g_{k,i}
 g_{k,j}
\frac{\partial^2  F }{\partial x_k^2}\circ  g +
\sum_{k<l} 
( g_{k,i}
 g_{l,j}+g_{l,i}g_{k,j})
\frac{\partial^2  F }{\partial x_k\partial x_l}\circ  g.
\]
Let $M$ be the automorphism of $\P(\SS)$ given by the linear forms 
\[
l_{i,j} = \displaystyle\sum_{k} g_{k,i}g_{k,j} z_{k,k}+\sum_{k<l} (g_{k,i}g_{l,j}+g_{l,i}g_{k,j}) z_{k,l}.
\]
Then, $h_G = M\circ h_F\circ g$. Since $g = h_F^{-1}\circ h_G$, we deduce that $M$ restricted to $h_F(\P^n)$ is the identity. Using that $h_F(\P^n)$ is nondegenerate, we get that $M = \mathrm{Id}$. Therefore,  for every $i,j$  we get that $l_{i,j} = \lambda z_{i,j}$ for $\lambda\in\K^{*}$. For $l_{i,i}=\lambda z_{i,i}$ we deduce that $g_{k,i}=0$ and $g_{i,i}=g_{k,k}$ for $i\neq k$. We conclude that $g = \mathrm{Id_n}$ and $F=G$.
\end{proof}

\para 

\subsection{Reconstruction algorithm for $n$ even}\label{subsec:alg d=4}

Once we have proven that $H_{4,n}$ is birational, our next goal is to find an effective algorithm for recovering $F$ from $H_{4,n}(F)$. We present an algorithm for the case where $n$ is even. 

\para

\begin{center}
\textsf{First step: computation of the Veronese variety $\mathcal{V}_X$}
\end{center}

\para

  Given $X$ a generic element in the image of $H_{4,n}$, 
  the first step of the algorithm is to
 compute the unique Veronese variety $\mathcal{V}_X$ containing $X$.   Let $m+1$ be the number of quadratic forms generating the ideal of $X$. Since $n\geq 2$, we have that 
$m\geq 2$. Among these quadrics, $m$ of them generate the ideal of $\mathcal{V}_X$ containing $X$. 

\para 

\begin{lemma}\label{lemma:recover vero}
 Let $X\in H_{4,n}(U)$ and let $I_2(X)$ be the vector space of quadrics containing $X$. Let  $W$ be a vector subspace of $I_2(X)$ of codimension $1$. Then, $\bbV(W)$ is a Veronese variety if and only if $\bbV(W)$ has linear syzygies.
\end{lemma}
\begin{proof}
By \cite{Syz1} Theorem 2.2, the Veronese variety has linear syzygies. 
Conversely, let $W = \langle \alpha_1,\ldots,\alpha_m\rangle\subseteq I_2(X) $ be a subspace of codimension $1$  with linear syzygies.  
By Proposition \ref{prop:uniq vero} there exists a unique Veronese variety $\mathcal{V}_X$ containing $X$. Let $W'\subset I_2(X)$ be the subspace generated by the quadrics containing $\mathcal{V}_X$, which has codimension $1$ in $I_2(X)$. Assume that $W\neq W'$. Then, there exists $q\in W$ such that $q+W'=I_2(X)$. Therefore, $\mathcal{V}_X\cap \mathbb{V}(W)=\mathcal{V}_X\cap \mathbb{V}(q) $. By Lemma \ref{lemma:qua vero}, we conclude that  $\mathcal{V}_X=\mathbb{V}(W)$. 
\end{proof}

\para

From the above result, we deduce that in order to find $\mathcal{V}_X$ it is enough to find $W\subset I_2(X)$ as in Lemma \ref{lemma:recover vero}. Let $J$ be the ideal of $\mathcal{V}_X$. Then, the minimal free resolution of $J$ looks like 
\[
0\xleftarrow{\hspace*{0.5cm}} J\xleftarrow{\hspace*{0.5cm}}R(-2)^{a}\xleftarrow{\hspace*{0.5cm}}R(-3)^{b}\xleftarrow{\hspace*{0.5cm}}\cdots,
\]
where $R$ is the coordinate ring of $\P(\SS)$.

\para 

\begin{lemma}
 Let $X\in H_{2,n}(U)$ and let $I$ be its ideal. Then 
$\beta_{0,j} = 0$ for $j\neq 2$, $\beta_{0,2} = a+1$, and
 $\beta_{1,3}= b$, where $\beta_{i,j}$ are the graded Betti numbers of $I$.
\end{lemma}
\begin{proof}
Since $X$ has codimension $1$ in $\mathcal{V}_X$, and it is the image through the Hessian map of a degree $4$ hypersurface, we deduce that  $I$ is generated by $a+1$ quadrics and hence, $\beta_{0,j} = 0 $ for $j\neq 2$ and $\beta_{0,2}=a+1$. Let $q_0,\ldots,q_a$ be generators of $I$ such that $q_1,\ldots,q_a$ generate the ideal of $\mathcal{V}_X$.
Now, since $X$ is contained in $\mathcal{V}_X$, we have that  $\beta_{1,3}\geq b$. Assume that $\beta_{1,3}>b$. This implies that there exists a linear sygygy of the form 
$
\sum_{i=0}^a l_iq_i,
$
where $l_i$ are linear forms and $l_0$ is non-zero. In particular, we get that $X\cap D(l_0)=\mathbb{V}(q_1,\ldots,q_a)\cap D(l_0)=\mathcal{V}_X\cap D(l_0)$. Since $\mathcal{V}_X$ is nondegenerate, this intersection is non empty and we get $X=\mathcal{V}_X$. Since $X$ has lower dimension that $\mathcal{V}_X$ we conclude that $\beta_{1,3}=b$.
\end{proof}

Let $I $ be the ideal of $X$. From the previous lemma we deduce that $I$ has a  minimal free resolution of the form
\begin{equation}\label{eq:min res}
0\xleftarrow{\hspace*{0.5cm}} I\overset{\pi_1}{ \xleftarrow{\hspace*{0.5cm}}} R(-2)^{a+1}\overset{M}{\xleftarrow{\hspace*{0.5cm}}}R(-3)^{b}\oplus F\xleftarrow{\hspace*{0.5cm}}\cdots,
\end{equation}
where $F$ is a free module, and $M$ is a block matrix 
of the form
\begin{equation}\label{eq:block matrix res}
M = \left(
\begin{array}{c|c}
\begin{array}{c}
M_1\\  \cline{1-1} 0
\end{array}
&M_2
\end{array}
\right).
\end{equation}
Here $M_1$ is an $a\times b$ matrix whose entries are linear forms. By Lemma \ref{lemma:recover vero}, the vector subspace $W$ of $I_2(X)$ generated by the first $a$ coordinates of $\pi_1$ generates the ideal $J$.
In order to find generators of $J$, it is enough to find a free resolution of the form \eqref{eq:min res}.
We present a method for obtaining such a resolution using linear algebra. Assume we are given a minimal free resolution of $I$ 
\[
0\xleftarrow{\hspace*{0.5cm}} I\overset{\pi_2}{\xleftarrow{\hspace*{0.5cm}}} R(-2)^{a+1}\overset{N}{\xleftarrow{\hspace*{0.5cm}}}R(-3)^{b}\oplus F\xleftarrow{\hspace*{0.5cm}}\cdots.
\]
We obtain the following isomorphisms of sequences

\[
\begin{tikzcd}
0 & I \arrow[l] \ar[d,-,double equal sign distance,double]  & R(-2)^{a+1} \arrow[l, "\pi_2"'] \arrow[d, "g"']               & R(-3)^b\oplus F \arrow[l, "N"'] \arrow[d, "h_1\oplus h_2"]                     
\\
0 & I \arrow[l] \ar[d,-,double equal sign distance,double] & R(-2)^{a+1} \arrow[l, "\pi_1"] \ar[d,-,double equal sign distance,double]  & R(-3)^b\oplus F \arrow[l, "M"] \arrow[d, "h_1^{-1}\oplus h_2^{-1}"]  
\\
0 & I \arrow[l]                                & R(-2)^{a+1} \arrow[l, "\pi_1"]                                & R(-3)^b\oplus F \arrow[l, "M\circ(h_1\oplus h_2)"]        
\end{tikzcd}
\]
where the first isomorphisms $g$ and $h_1\oplus h_2$ exist by the uniqueness of minimal free resolutions and $M$ is as in \eqref{eq:block matrix res}. In particular, $M\circ (h_1\oplus h_2)$ is a matrix of the same form as \eqref{eq:block matrix res}. Hence, there exists a $g\in\mathrm{GL}(a+1)$ such that 
\[
\begin{tikzcd}
0 & I \arrow[l] \ar[d,-,double equal sign distance,double]  & R(-2)^{a+1} \arrow[l, "\pi_2"'] \arrow[d, "g"']               & R(-3)^b\oplus F \arrow[l, "N"'] \ar[d,-,double equal sign distance,double]                  
\\
0 & I \arrow[l]  & R(-2)^{a+1} \arrow[l, "\pi"]  & R(-3)^b\oplus F \arrow[l, "M'"] 
\end{tikzcd}
\]
is an isomorphism to a complex of the form \eqref{eq:min res}. In particular, the first $a$ entries of the map $\pi_2\circ g^{-1}$ generate the ideal of $\mathcal{V}_X$. The computation of $g$ is achieved by solving the linear system in the entries of $g$ given by requiring the first $b$ entries of the last row of $g\circ N$ to vanish. 

\para

Let us illustrate the previous procedure by an example.

\para 

\begin{example}\label{ex:deg 4 1}
Consider the hypersurface in $\P^2$ given by the polynomial 
\[
F = x_{0}^{3}x_1+x_{1}^{3}x_{2}+x_{2}^{3}x_0 +3x_{0}^{2}x_1x_2.
\]
The ideal of $X := H_{4,2}(F)\subset \P^5$ is
\begin{equation}\label{eq:ideal X}\hspace*{-4mm}\displaystyle \begin{array}{l}
I  =  
\langle
z_0^2-z_0z_3+z_2z_3+4z_0z_4-4z_2z_4-4z_3z_4+4z_1z_5+\frac{1}{2}z_2z_5-4z_4z_5-3z_5^2,   \\ 
\noalign{\vspace*{1mm}}   
    z_0z_1+\!z_2z_4-\!\frac{3}{2}z_0z_5+\!\frac{1}{2}z_2z_5+z_3z_5  , z_0z_2+4z_0z_4-\!4z_2z_4-\!4z_3z_4+2z_1z_5-\!2z_4z_5-\!z_5^2,  \\  \noalign{\vspace*{1mm}}  
    z_1^2+\frac{1}{4}z_0z_3-\frac{1}{4}z_2z_3-z_1z_4-2z_1z_5-\frac{1}{8}z_2z_5+z_4z_5+z_5^2 ,   
    z_1z_2-\frac{1}{2}z_0z_5-\frac{1}{2}z_2z_5 ,   \\\noalign{\vspace*{1mm}}  
    z_2^2+4z_0z_4-4z_2z_4-4z_3z_4+z_5^2 ,
    z_1z_3+z_2z_4-\frac{1}{4}z_5^2
    \rangle
\end{array}
\end{equation}

The first part of the free resolution of $I$ is 
\[
0\longleftarrow R\overset{A_0}{\longleftarrow} R(-2)^7\overset{A_1}{\longleftarrow} R(-3)^8\oplus R(-4)^6,
\]
where $R=\K[z_0,z_1,z_2,z_3,z_4,z_5]$.
The submatrix $N$  of $A_1$ with linear entries is
\[\tiny
 \begin{pmatrix}
-z_2-4z_4  &      \frac{1}{2}z_5  &      0  &      -4z_4  &      0  &      z_1-\frac{3}{2}z_5  &      -\frac{1}{4}z_3  &      0 \\ \noalign{\vspace*{1mm}} 
-2z_5  &      4z_4  &      \frac{1}{2}z_5  &      0  &      -4z_4  &      -z_0-4z_4  &      -z_1+z_4+\frac{1}{2}z_5  &      -z_2-4z_4  \\\noalign{\vspace*{1mm}} 
0  &      2z_5  &      -z_2  &      0  &      0  &      -4z_5  &      z_0-z_3  &      -2z_5  \\ \noalign{\vspace*{1mm}} 
z_0-z_3+8z_4  &      -z_1+\frac{1}{2}z_5  &      \frac{1}{4}z_3  &      -z_2+8z_4  &      -\frac{1}{2}z_5  &      \frac{1}{2}z_5  &      \frac{1}{4}z_3+\frac{1}{8}z_5  &      z_1-\frac{3}{2}z_5  \\ \noalign{\vspace*{1mm}} 
4z_5  &      z_0-4z_4  &      z_1-z_4-\frac{3}{2}z_5  &      2z_5  &      -z_2+4z_4  &      -z_3+4z_4-\frac{1}{2}z_5  &      \frac{1}{2}z_5  &      -z_3+4z_4  \\ \noalign{\vspace*{1mm}} 
z_3-4z_4+\frac{1}{2}z_5  &      0  &      -\frac{1}{4}z_3-\frac{1}{8}z_5  &      z_0-4z_4  &      z_1-\frac{1}{2}z_5  &      0  &      0  &      \frac{1}{2}z_5   \\ \noalign{\vspace*{1mm}} 
2z_5  &      -4z_4  &      -\frac{1}{2}z_5  &      0  &      4z_4  &      z_0+4z_4  &      z_1-z_4-\frac{1}{2}z_5  &      z_2+4z_4  
\end{pmatrix}
\]
Choosing 
\[
g = \begin{pmatrix}
1&0&0&0&0&0&0\\
0&1&0&0&0&0&0\\
0&0&1&0&0&0&0\\
0&0&0&1&0&0&0\\
0&0&0&0&1&0&0\\
0&0&0&0&0&1&0\\
0&1&0&0&0&0&1\\
\end{pmatrix},
\]
we get that the last row of $g\cdot N$ is zero. Therefore,  the six generators of the Veronese variety correspond to the first six entries of the matrix $A_0\circ g^{-1}$. We conclude that the ideal of the unique Veronese variety containing $X$ is the zero locus of the ideal $J $ given by 
\begin{equation}\label{eq:ideal vero}\begin{array}{c} J =  \langle
z_0^2-z_0z_3+z_2z_3+4z_0z_4-4z_2z_4-4z_3z_4+4z_1z_5+\frac{1}{2}z_2z_5-4z_4z_5-3z_5^2,   \\ \noalign{\vspace*{1mm}} 
  z_0z_1-z_1z_3-\frac{3}{2}z_0z_5+\frac{1}{2}z_2z_5+z_3z_5+\frac{1}{4}z_5^2,   \\ \noalign{\vspace*{1mm}} 
  z_1^2+\frac{1}{4}z_0z_3-\frac{1}{4}z_2z_3-z_1z_4-2z_1z_5-\frac{1}{8}z_2z_5+z_4z_5+z_5^2 ,   \\ \noalign{\vspace*{1mm}} 
  z_0z_2+4z_0z_4-4z_2z_4-4z_3z_4+2z_1z_5-2z_4z_5-z_5^2 ,   \\ \noalign{\vspace*{1mm}} 
  z_1z_2-\frac{1}{2}z_0z_5-\frac{1}{2}z_2z_5, 		\\ \noalign{\vspace*{1mm}}   
  z_2^2+4z_0z_4-4z_2z_4-4z_3z_4+z_5^2\rangle
\end{array}
\end{equation}
\end{example}

\begin{center}
\textsf{Second step: computation  of an isomorphism $\varphi:\mathcal{V}_X\rightarrow \P^n.$}
\end{center}

\para

 The next step of our algorithm is to find an isomorphism $\varphi:\mathcal{V}_X\rightarrow \P^n$,
which is the inverse of a Veronese embedding $v_2:\P^n\rightarrow\mathcal{V}_X$. We set $L\simeq \varphi^{*}\mathcal{O}_{\P^n}(1)$.

\para 

\begin{lemma}\label{lemma:line bundle}
 For $n$ even, $L\simeq \omega_{\mathcal{V}_X}\otimes \O_{\P(S^2V)}(k+1)|_{\mathcal{V}_X}$ where $2k=n$.
\end{lemma}
\begin{proof}
Since $\varphi$ is the inverse of a Veronese embedding, we get that  $L^{\otimes 2}=\O_{\P(S^2V)}(1)|_{\mathcal{V}_X}$. Since $\varphi$ is an isomorphism, we have that $\varphi^{*}\omega_{\P^n} = \omega_{\mathcal{V}_X}$ and hence, $L^{\otimes -n-1}\simeq \omega_{\mathcal{V}_X}$. Let $k\in\N$ such that $n=2k$. Then, 
\[
L = L^{\otimes n+2}\otimes L^{\otimes -n-1}= \O_{\P(S^2V)}(1)|_{\mathcal{V}_X}^{\otimes k+1}\otimes\omega_{\mathcal{V}_X}.
\]
\end{proof}

 In particular, $\varphi$ is given by a basis of the space of global sections of $\omega_{\mathcal{V}_X}\otimes \O_{\P(S^2V)}(k+1)|_{\mathcal{V}_X}$. 
Let $\iota$ be the inclusion of $\mathcal{V}_X$ in $\P^{N}=\P(S^2V)$.
Let $R=\K[z_0,\ldots,z_N]$ be the homogeneous coordinate ring of $\P^N$ and let  $J$ be the ideal of $\mathcal{V}_X$. 
In \cite{Syz2}, the Betti numbers
 of the Veronese embbeding given by $|\mathcal{O}_{\P^n}(2)|$ are computed. In particular, we deduce  that the minimal free resolution of $R/J$ is of the form 
 \begin{equation}\label{eq:free res1}
0\xleftarrow{\hspace*{5mm}}
R
\overset{A_1}{\xleftarrow{\hspace*{5mm}}}
P_1
\overset{A_2}{\xleftarrow{\hspace*{5mm}}}
\cdots 
\overset{A_r}{\xleftarrow{\hspace*{5mm}}}
 R(k-N)^{n+1}\xleftarrow{\hspace*{5mm}} 0.
\end{equation}
where $r$ is the codimension of $\mathcal{V}_X$ in $\P^N$.

\para

\begin{proposition}\label{prop:ext}
 For $n=2k$, the line bundle $L$ is the sheafification of the graded $R/J$--module 
 \begin{equation}\label{eq:mod sections}
 (R/J)^{n+1}/\mathrm{Im} \,A_r^{t}
 \end{equation}
\end{proposition}
\begin{proof}

By \cite[Corollary 7.12]{Hart}, we have that 
$
\omega_{\mathcal{V}_X}\simeq\iota^{*} \mathcal{E}\mathrm{xt}^r_{\P^N}(\iota_{*}\O_{\mathcal{V}_X},\omega_{\P^N}).
$
Using \cite[Proposition 7.7]{GW}, we get that 
\[
 \omega_{\mathcal{V}_X}\otimes \O_{\P(S^2V)}(k+1)|_{\mathcal{V}_X}\simeq \iota^{*} \mathcal{E}\mathrm{xt}^r_{\P^N}(\iota_{*}\O_{\mathcal{V}_X},\omega_{\P^N}(k+1))
\simeq \iota^{*} \mathcal{E}\mathrm{xt}^r(\iota_{*}\O_{\mathcal{V}_X},\O_{\P^N}(k-N)).
\]
Applying $\mathcal{H}\mathrm{om}(-,R[k-N])$ to \eqref{eq:free res1} we get the   complex
\[
0\longleftarrow R^{n+1}\overset{A_r^{\mathrm{t}}}{\longleftarrow }\cdots
\overset{A_2^{\mathrm{t}}}{\longleftarrow }R[k-N]\longleftarrow 0.
\]
We deduce that $\mathcal{E}\mathrm{xt}^r(\iota_{*}\O_{\mathcal{V}_X},\O_{\P^N}(k-N))$ is the sheafification of the $r$--th cohomology group of this   complex which is 
$
R^{n+1}/\mathrm{Im}A_r^{\mathrm{t}}.
$
The proposition follows by tensoring this module by $R/J$.
\end{proof}

Let $M$ denote the graded $R/J$--module in equation \eqref{eq:mod sections}.
From this proposition we deduce that the space of global sections of $L$ is the zero graded piece of  $M$, i.e. 
$
M_0 = \K^{n+1}.
$
Let $e_0,\ldots,e_n$ be a basis of $M_0$ and let $\mathcal{U} =  D(z_0)\cap \mathcal{V}_X$. We get an inclusion 
\[
M_0=H^0(\mathcal{V}_X, \omega_{\mathcal{V}_X}(k+1))\hookrightarrow H^0(U,\omega_{\mathcal{V}_X}(k+1))=(M_{z_0})_0,
\]
where $M_{z_0}$ denotes the localization of $M$ by $z_0$. On the other hand, since $L$ is trivial on $\mathcal{U} $, we get an isomorphism 
$
\phi:(M_{z_0})_0\rightarrow(( R/J)_{z_0})_0.
$
This allows us to define the map 
\[\begin{array}{cccc}
\varphi:& U&\longrightarrow& \P^n\\
 & z&\longmapsto&[\phi(e_0)(z):\cdots:\phi(e_n)(z)].
\end{array}
\]
We illustrate the previous reasoning by an example.

\para

\begin{example}\label{ex:deg 4 2}
Consider the Veronese surface $\mathcal{V}$ in $\P^5$ given by the ideal $J$ of equation 
 \eqref{eq:ideal vero}.
The free resolution of $R/J$ is
\[
0\rightarrow R\rightarrow R(-2)^6\rightarrow R(-3)^8\overset{A}{\rightarrow} R(-4)^3\rightarrow 0,
\]
where 
\[
A = \begin{pmatrix}
  z_2+4z_4     & -\frac{1}{2}z_5     & 4z_4      \\
  2z_5     &    z_2     &     0       \\
  -z_1+\frac{3}{2}z_5   & -\frac{1}{4}z_3     & \frac{1}{2}z_5   \\
  z_0+4z_4     & -z_1+z_4+\frac{1}{2}z_5 & 4z_4\\
  z_3-4z_4     & -z_1+z_4+\frac{3}{2}z_5& z_2-4z_4    \\
  4z_5     &    z_0-z_3     &  2z_5     \\
  -\frac{1}{2}z_5     & \frac{1}{4}z_3+\frac{1}{8}z_5 & -z_1+\frac{1}{2}z_5 \\
  z_3-4z_4+\frac{1}{2}z_5& \frac{1}{2}z_5     &  z_0-4z_4    
\end{pmatrix}
\]
By Proposition \ref{prop:ext},  $L\simeq \omega_{\mathcal{V}}\otimes\mathcal{O}_{\P^5}(2)|_{\mathcal{V}}$ is the sheafification of the graded module
\[
M:=(R/J)^{3}/\langle \text{rows of } A \rangle.
\]
The space of global sections of $L$ is isomorphic to the zero graded piece of $M$, which is $\K^3$. Let $e_0,e_1,e_2$ be the usual basis of this space. Now consider the principal affine open subset $ D(z_1)$. Then, 
\[
H^0( D(z_1),L) = ((M)_{z_1})_0 = ((R/J)_{z_1})_0^{\oplus 3}/\langle \text{rows of }A_{z_1}\rangle\overset{\phi}{\simeq} ((R/J)_{z_1})_0,
\]
where $A_{z_1}=\frac{1}{z_1}A$.
One can check that in $\langle \text{rows of }A_{z_1}\rangle$ we have the elements
\[
e_0+\frac{1}{2}\left(\frac{z_0}{z_1}+\frac{z_2}{z_1}\right)e_1 \text{ and } -e_2+ \frac{1}{8}\left(-2\frac{z_0}{z_1}+6\frac{z_2}{z_1}+4\frac{z_3}{z_1}+\frac{z_5}{z_1}\right)e_1.
\]
Therefore, we get that 
\[
\phi(e_0) = -\frac{1}{2}\left(\frac{z_0}{z_1}+\frac{z_2}{z_1}\right)\phi(e_1) \text{ and } \phi(e_2) =  \frac{1}{8}\left(-2\frac{z_0}{z_1}+6\frac{z_2}{z_1}+4\frac{z_3}{z_1}+\frac{z_5}{z_1}\right)\phi(e_1).
\]
Thus, we get an isomorphism  
\[
\begin{array}{ccc}\displaystyle
 D(z_1)\cap \mathcal{V}& \overset{\varphi}{\longrightarrow }&\P^2 
\\
   \displaystyle
   \left[\frac{z_0}{z_1}:\cdots:\frac{z_5}{z_1}\right]& \longmapsto & \displaystyle\left[-4\left(\frac{z_0}{z_1}+\frac{z_2}{z_1}\right):8:-2\frac{z_0}{z_1}+6\frac{z_2}{z_1}+4\frac{z_3}{z_1}+\frac{z_5}{z_1}\right].
\end{array}
\]
Note that once we have the map $\varphi$, we can compute a parametrization of $\mathcal{V}$. This parametrization is the inverse of $\varphi$ and is given by a basis $\{F_0,\ldots,F_5\}$ of $S^2\K^3$. In particular, we get that for generic $[x,y,z]\in\P^2$ we must have that 
\[
 [-4(F_0+F_2):8F_1:-2F_0+6F_2+4F_3+F_5]=[x:y:z].
\]
This defines some linear equations in the coefficients of $F_0,\ldots,F_5$. A generic solution of such linear system of equations will give a parametrization of $V$. In our example, one can check that the polynomials 
\[
x_{0}^{2},x_{1}^{2},-x_{0}^{2}-2x_0x_{1},x_{2}^{2},x_0x_2,8x_{0}^{2}+12x_0x_{1}+8x_{1}x_{2}-4x_{2}^{2}
\]
provide a parametrization of the Veronese variety.
\end{example}

\para 

\begin{remark}\label{rem:para}
Assume that we are given the ideal  $J$ of a Veronese variety of even dimension $n$ that it is the image of quadratic Veronese embedding of line bundle. The previous study allows us to compute an isomorphism between the Veronese variety and $\P^n$. As shown in Example \ref{ex:deg 4 2}, using linear algebra we get a method for computing a parametrization of the Veronese variety from its ideal.
\end{remark}

\para 

\begin{remark}
    The computation of the isomorphism $\varphi:\mathcal{V}_X\rightarrow \P^n$ is the only step of the algorithm where we  use that $n$ is even. 
    For $n$ odd, the canonical bundle of $\P^n$ has even degree. As a consequence, we can not write $L$ by means of $\omega_{\mathcal{V}_X}$.  One needs to find a $\mathcal{O}_{\P^N}$--module whose restriction to $\mathcal{V}_X$ is a line bundle of odd degree.
\end{remark}
\para

\begin{center}
\textsf{Third step: computation of $\varphi(X)$}
\end{center}

\para

Once we have computed an isomorphism $\varphi:\mathcal{V}_X\rightarrow \P^n$, we compute the image of $X$ through $\varphi$. This image is a  hypersurface defined by a degree $4$ polynomial $G$. Let $F\in\P(S^4V)$ be such that $H_{4,n}(F)=X$. In other words, $F$ is the polynomial we want to compute. Then, the composition $\varphi\circ h_F$ is an automorphisms of $\P^n$ that sends  $\bbV(F)$ to $\bbV(G)$. In particular, we get that there exists $g\in\mathrm{PGL}(n+1)$ such that $G = g\cdot F$. Moreover, by Lemma \ref{lemma:n=1 2} we get the following commutative diagram
\[\begin{tikzcd}
\P^n \arrow[rr, "h_F"] \arrow[dd, "(g^{\mathrm{t}})^{-1}"'] &  & \mathcal{V}_X \arrow[dd, "\rho(g^{\mathrm{t}})^{\mathrm{t}}"] \arrow[lldd, "\varphi"'] \\
                                                            &  &                                                                           \\
\P^n \arrow[rr, "h_G"']                                     &  & h_G(\P^n)                                                                
\end{tikzcd}.
\]
Therefore, we get the that $\rho(g^{\mathrm{t}})^{\mathrm{t}}=h_G\circ \varphi$. Note that the representation $\rho$ is injective and the preimages through $\rho$ can be computed by a solving  linear system of equations. In particular, we can recover $g$ from the composition $h_G\circ \varphi$. Finally, the polynomial $F$ is computed by applying $g^{-1}$ to $G$. 

\para

\begin{center}
\textsf{The algorithm}
\end{center}

\para

Summarizing the previous reasoning, we outline the algorithm

\para 

\begin{algorithm}[H]
\caption{}\label{alg:2}

\vspace*{2mm}

\noindent \textsf{Input:}  the ideal $I$ defining a variety inside the image of  $H_{4,n}$ for $n=2k$

\vspace*{1mm}

\noindent \textsf{Output:}   the unique  $F\in \P(S^4V)$ such that $H_{4,n}(F) = \bbV(I)$.
\begin{enumerate}
\item  Compute the unique Veronese variety $\mathcal{V}$ containing $\bbV(I)$.
\item Determine an isomorphism $\varphi:\mathcal{V}\rightarrow \P^n$.
\item Compute the quartic form $G$ defining $\varphi(\bbV(I)).$
\item Determine $ g^{\mathrm{t}} =\rho^{-1}((h_G\circ \varphi)^{\mathrm{t}})$. Return, $F= g^{-1}\cdot G$.
\end{enumerate}
\end{algorithm}

\para 

\begin{example}
Consider the variety $X$ given by the ideal $I$ in equation \eqref{eq:ideal X}. By Example \ref{ex:deg 4 1}, the ideal $J$ of the unique Veronese variety $\mathcal{V}_X$ is given by \eqref{eq:ideal vero}. By Example \ref{ex:deg 4 2}, we have an isomorphism $\varphi:\mathcal{V}_X\rightarrow\P^2$ sending $[z_0:\cdots:z_5]$ to $$[-4(z_0+z_2):8z_1:-2z_0+6z_2+4z_3+z_5].$$
The image of $X$ through this map is given by the polynomial \[
G =
2x_{0}^{4}-8x_{0}^{3}x_1+6x_{0}^{2}x_{1}^{2}+x_0x_{1}^{3}+4x_{0}^{3}x_{2}-48x_{0}^{2}x_1x_2+12x_0x_{1}^{2}x_{2}-96x_0x_1x_{2}^{2}-64x_1x_{2}^{3}.\]
One can check that, modulo $J$  the composition $h_G\circ\varphi$ is given by the matrix 
\begin{equation}\label{eq:mat ex}
\begin{pmatrix}
0  &    0  &    1  &    0  &    -4  &    4 \\
0  &    -1  &    0  &    2  &    0  &    0 \\
1  &    0  &    0  &    0  &    0  &    0 \\
0  &    0  &    0  &    0  &    -4  &    8 \\
0  &    0  &    0  &    4 &    0  &    0 \\
0  &    0  &    0  &    0  &    0  &    16
\end{pmatrix}.
\end{equation}
In this case, the representation $\rho$ is given by 
\[
\begin{pmatrix}
a_0&a_1&a_2\\
b_0&b_1&b_2\\
c_0&c_1&c_2
\end{pmatrix}
\longmapsto
\begin{pmatrix}
a_0^2 &  a_0b_0 &  b_0^2 &  a_0c_0 &  b_0c_0 &  c_0^2 \\
2a_0a_1 &  a_1b_0+a_0b_1 &  2b_0b_1 &  a_1c_0+a_0c_1 &  b_1c_0+b_0c_1 &  2c_0c_1 \\
a_1^2 &  a_1b_1 &  b_1^2 &  a_1c_1 &  b_1c_1 &  c_1^2 \\
2a_0a_2 &  a_2b_0+a_0b_2 &  2b_0b_2 &  a_2c_0+a_0c_2 &  b_2c_0+b_0c_2 &  2c_0c_2 \\
2a_1a_2 &  a_2b_1+a_1b_2 &  2b_1b_2 &  a_2c_1+a_1c_2 &  b_2c_1+b_1c_2 &  2c_1c_2 \\
a_2^2 &  a_2b_2 &  b_2^2 &  a_2c_2 &  b_2c_2 &  c_2^2
\end{pmatrix}.
\]
One can check that the unique $g\in\mathrm{PGL}(3)$ such that $\rho(g^{\mathrm{t}})^\mathrm{t}$ equals the matrix \eqref{eq:mat ex} is
\[
g = \begin{pmatrix}
0&1&-2\\-1&0&0\\0&0&-4
\end{pmatrix}.
\]
We conclude that $H_{4,2}^{-1}(X)$ is the polynomial 
\[
 g^{-1}\cdot G = -256x_{0}^{3}x_1-768x_{0}^{2}x_{1}x_{2}-256x_{1}^{3}x_2-256x_0x_{2}^{3}.
\]
Note that this polynomial, 
up to multiplication by constants, coincides with the polynomial we started with in Example \ref{ex:deg 4 1}.
\end{example}

\para 

\section{Conclusions}

We have introduced the Hessian correspondence $H_{d,n}$ and the Hessian correspondence $H_{d,n,k}$ for polynomials with Waring rank $k$. We have studied $H_{d,n,k}$ for $k\leq n+1$ and $H_{d,n}$ for hypersurfaces of degree $3$ and $4$.
In the case of hypersurfaces of Waring rank $k\leq n+1$ we have shown that, for $d$ even, $H_{d,n,k}$ is birational onto its image, whereas for $d$ odd, it is genericaly finite of degree $2^{k-1}$.
In the case $d=3$, we have seen that  $H_{3,1}$ is a two to one map and that $H_{3,n}$ is birational onto its image for $n\geq 2$. We have introduced the variety of $k$--gradients $\phi_k$ and we have computed its irreducible components as well as their dimensions.  For the case $d=4$, we have seen that $H_{4,n}$ is always birational onto its image. 
From a computational point of view, we have provided effective algorithms for recovering $F$ from its Hessian variety when either $d=3$ and $n\geq 1$, or $d=4$ and $n=1$, or  $d=4$ and $n$ even. As future work,  the natural question of deriving a recovery algorithm for $d=4$ and $n$ odd appears. Other possible open problem is the study of the map $H_{d,n}$ for $d\geq 5$ or the map $H_{d,n,k}$ for $k>n+1$. 

\vspace*{3mm}
\subsection*{Acknowledgements}
I want to express my gratitude to Daniele Agostini for all his help and guidance. I am also grateful to Fulvio Gesmundo, Ángel David Rios Ortiz, Pierpaola Santarsiero and Josef Schicho for helpfull and inspiring discussions on the topic. I received the support of a fellowship from the "la Caixa" Foundation (ID 100010434). The fellowship code is LCF/BQ/EU21/11890110.

\addcontentsline{toc}{section}{References}\label{sec:references}
\bibliographystyle{plain}

\para 

\subsection*{Affiliation}

\para 

\noindent \textsc{Javier Sendra--Arranz,\\ Eberhard Karls Universität Tübingen (Tübingen) and Max Planck Institute for Mathematics in the Science (Leipzig)}\\
\hfill {\tt sendra@math.uni-tuebingen.de}
\para 

\end{document}